\documentclass[preprint,12pt]{elsarticle}
\usepackage{amsmath,amsthm,amssymb,amsfonts,mathdots,graphicx,algorithm,algpseudocode}
\usepackage{a4wide}
\usepackage{csquotes}
\usepackage{bm}
\usepackage{url}
\usepackage{tikz}
\usetikzlibrary{decorations.markings,shapes.geometric,arrows}
\usepackage{mathrsfs}
\usepackage{xfrac}
\usepackage{tabularx}
\usepackage{caption}
\usepackage{multirow}
\usepackage{arydshln}

\tikzstyle{decision} = [diamond, minimum width=3cm, minimum height=1cm, text centered, draw=black]
\tikzstyle{leaf} = [circle, radius=3pt, draw=black, fill=black]
\tikzstyle{arrow} = [thick,->,>=stealth]

\journal{arXiv}

\newcommand{\argmax}{\operatornamewithlimits{argmax}}
\newcommand{\argmin}{\operatornamewithlimits{argmin}}
\newcommand{\prob}{\text{Prob}}
\newcommand{\defeq}{\stackrel{\text{def}}{=}}
\newcommand{\setR}{\mathbb{R}}
\newcommand{\setN}{\mathbb{N}}

\newcommand{\eucNorm}[1]{\| #1 \|}

\newcommand{\RN}[2]{\frac{\mathrm{d}#1}{\mathrm{d}#2}}
\newcommand{\randHat}{\boldsymbol{\hat{\theta}}}
\newcommand{\supp}{\mathop{\mathrm{supp}}}

\newtheorem{thm}{Theorem}[section]

\newtheorem{lemma}[thm]{Lemma}
\newtheorem{cor}{Corollary}[thm]

\theoremstyle{definition}
\newtheorem{defi}[thm]{Definition}

\allowdisplaybreaks
\begin{document}
\begin{frontmatter}
\title{A taxonomy of estimator consistency on discrete estimation problems}
\author[otzma,monash]{Michael Brand}
\ead{research@OtzmaAnalytics.com}
\author[monash]{Thomas Hendrey}
\ead{thomas.hendrey@monash.edu}
\address[otzma]{Otzma Analytics Pty Ltd, Bentleigh East, VIC 3165, Australia}
\address[monash]{Faculty of IT (Clayton), Monash University, Clayton, VIC 3800, Australia}
\begin{abstract}
We describe a four-level hierarchy mapping both all discrete estimation
problems and all estimators on these problems, such that the hierarchy
describes each estimator's consistency guarantees on each problem class.
We show that no estimator is consistent for all estimation problems, but that
some estimators, such as Maximum A Posteriori, are consistent for the
widest possible class of discrete estimation problems. For Maximum Likelihood
and Approximate Maximum Likelihood estimators we show that they do not
provide consistency on as wide a class, but define a sub-class of problems
characterised by their consistency. Lastly, we show that some popular
estimators, specifically Strict Minimum Message Length, do not provide
consistency guarantees even within the sub-class.

\end{abstract}

\end{frontmatter}

\section{Introduction}\label{Sec:Intro}
Consistency has been studied extensively as a desirable property of estimators by proponents of diverse approaches to point estimation. For instance there is an extensive literature investigating assumptions sufficient to assure the consistency of maximum likelihood estimators of which  \cite{Doob1934} is the first, \cite{Wald1949} is the classic source, \cite{Perlman1972} provides a good summary and \cite{SeoLindsay2013} is a more recent development. On the Bayesian side the related issue of ``posterior consistency'' has an extensive literature beginning with \cite{Doob1949} and summarised in \cite{ghosal1997review}. This was related to the consistency of Bayes estimators in \cite{Schwartz1965}, and consistency of Bayes estimators was more directly analysed in \cite{DiaconisFreedman1986}. More recently Dowe \cite{DoweBaxterOliverWallace1998,Dowe2011mml} conjectured that under certain conditions only Bayesian estimation methods would be consistent.

The statistical notion of consistency is typically discussed in the context of a given estimation method and a single problem or class of problems (as in \cite[][Section 6]{Schwartz1965} and \cite{dowe1997resolving}).
By contrast, in this paper we show that estimators fall into natural classes
in terms of the properties they require of an estimation problem in order to
ensure consistency. This allows us to map any estimator according to the
class it belongs to, and, in turn, also to map any estimation problem
according to the same classification, by the properties it exhibits.

Specifically, we define three nested classes of estimation problems. The
widest class is those including all estimation problems. (Here, ``all
estimation problems'' refers to all discrete estimation problems, i.e.
to estimation problems for which both the parameter space and the observation
space are countable. Throughout this paper, this will be the scope of
the discussion.) Smaller than the class of all problems is the
class of problems having \emph{convergent likelihood ratios}. And the smallest
is the class of problems having \emph{distinctive likelihoods}. An estimator
is classified according to the largest of these classes within which it
guarantees consistency for any problem. Thus, we describe estimator
consistency as a hierarchy: at level zero of the hierarchy
are the estimators consistent for
all problems, at level one those that are not at level zero but are
consistent for all problems with convergent likelihood ratios (We call
such estimators \emph{properly consistent}), at level two are those
estimators that are not at level one but are consistent for all problems
with distinctive likelihoods (We call such estimators
\emph{likelihood consistent}), and at level three are
those estimators for which not even
the property of distinctive likelihoods is a guarantee of consistency.

In the paper, we provide several characterisations to each of the problem
classes, taking both Bayesian and frequentist approaches. We then provide
examples for popular estimators and for wide estimator classes mapped into
each level, proving both their consistency within the required problem
class and their inconsistency within the next-largest problem class.

Specifically, we show that
\begin{enumerate}
\item No estimator resides in level zero of the hierarchy, and that, in fact,
no estimator can be consistent for any estimation problem outside the class
of those with convergent likelihood ratios,
\item Maximum A Posteriori (MAP) resides in level one of the hierarchy, being
consistent over all problems that have any consistent estimators, and so do
some frequentist estimators, but not Maximum Likelihood (ML),
\item Maximum Likelihood (ML) estimation resides in level two of the
hierarchy, being consistent for all problems with distinctive likelihoods, and
\item Strict Minimum Message Length (SMML) estimation resides in level three of
the hierarchy, not guaranteeing consistency even for problems with
distinctive likelihoods.
\end{enumerate}

The last result, namely the result regarding SMML, resolves in the negative a
conjecture by Dowe \cite{DoweBaxterOliverWallace1998, Dowe2011mml} to the
effect that Minimum Message Length estimators are consistent on more general
classes of problems than alternatives such as Maximum Likelihood, and
frequentist estimation methods in general.

The remaining parts of this paper are organised as follows.
In Section~\ref{Sec:Def}, we provide general definitions regarding the
point estimation problem, precisely scoping our present results.
In Section~\ref{Sec:PropCon}, we define the class of estimation problems
with convergent likelihood ratios, the corresponding class of
properly consistent
estimators, provide examples, and show that no estimation problem without
convergent likelihood ratios has consistent estimators at all.
In Section~\ref{S:LC}, we define the class of estimation problems with
distinctive likelihoods, the corresponding class of likelihood consistent
estimators, and provide examples.
Lastly, in Section~\ref{S:SMML}, we analyse the Strict Minimum Message Length
estimator, and show an estimation problem with distinctive likelihoods for
which it is provably not consistent.

\section{Definitions}\label{Sec:Def}
\subsection{Estimation Problems}\label{Sec:EstProb}
Point estimation concerns the problem of inferring a single point estimate $\hat{\theta}$ of an unknown parameter of a model given observations generated from the model. Point estimation problems are divided into \emph{small sample problems} (where the task is to make an estimate given a fixed, finite number of observations), and large sample problems (where the task is to provide estimates which have good properties in the limit of infinite observations). While consistency is a concern only in large sample problems, we first introduce small sample problems as they are the most natural context for defining most of the estimators we shall be interested in.

\begin{defi}\label{D:classsmallProblem}
A \emph{discrete, small sample, classical estimation problem} is a tuple $\langle X,\Theta,\mathscr{P}\rangle$ where:
\begin{itemize}
\item $X$ is called the \emph{observation space}, and is a finite product of countable subsets of some Euclidean space, representing possible finite sequences of observations that might be observed.
\item $\Theta$ is called the \emph{parameter space}, and is a countable subset of a Euclidean space representing possible values of the unknown parameter.
\item $\mathscr{P}=(P_\theta)_{\theta\in\Theta}$ is a set of distinct probability distributions over $X$ indexed by $\Theta$, where $P_\theta$ represents the probability distribution over the observations implied by the model with parameter value $\theta$.
\end{itemize}
\end{defi}

The intended interpretation is that we receive an observation (which is a term we will also use when describing a finite sequence of observations), from which we make our estimate, $\hat{\theta}$, of $\theta$. The observation is modelled as a random variable $\mathbf{x}$ whose distribution is $P_\theta$, which depends on the true value of the unknown parameter.

The classical approach to point estimation assumes no prior information regarding the value of the unknown parameter value $\theta$ except for its set of possible values. This is appropriate when one is evaluating ``frequentist'' approaches which make no use of any prior information, or when considering properties such as bias or variance of estimators. In order to evaluate Bayesian approaches in the same framework, one must assume that the prior is selected by the statistician when selecting an estimation method. This is incongruous with (mainstream) Bayesian philosophy, since according to this philosophy the prior distribution should capture the statistician's prior knowledge about the unknown parameter, and therefore the statistician is not free to choose this prior arbitrarily. Such a Bayesian, therefore, should treat the prior as given in the estimation problem, instead of something to be determined freely in choosing an estimation method. Therefore in this paper we shall consider what we call \emph{Bayesian} estimation problems:

\begin{defi}\label{D:bayessmallProblemtup}
A \emph{discrete, small sample, Bayesian estimation problem} is a tuple,\break
$\langle X,\Theta,\mathscr{P},\Pi\rangle$, where $X$, $\Theta$ and $\mathscr{P}$ are defined as in Definition \ref{D:classsmallProblem}, and $\Pi$ is a probability distribution (called the \emph{prior distribution}) over $\Theta$ such that $\Pi(\{\theta\})>0$  for all $\theta\in\Theta$.
\end{defi}

The addition of the prior distribution allows us to treat the value of the unknown parameter as a random variable $\boldsymbol{\theta}$, distributed according to $\Pi$. Note that we assume that the prior distribution is everywhere positive, reflecting the informal notion of $\Theta$ as the set of values that the statistician considers possible values for the unknown parameter. Other than this we make no assumptions about the form of the prior distribution.

The distributions in $\mathscr{P}$ can then be thought of as conditional distributions of $\mathbf{x}$, conditioned on different possible values of $\boldsymbol{\theta}$. As a result we can alternatively, and more simply, conceive of a large sample Bayesian estimation problem as being characterised by the joint distribution on the pair of random variables $(\boldsymbol{\theta},\mathbf{x})$. That is, the following definition is equivalent to Definition \ref{D:bayessmallProblemtup}, and is the one we shall mostly use:\footnote{Note that, when $\Theta$ is uncountable these two definitions are no longer always equivalent, since the conditional probability $\prob(\cdot|\boldsymbol{\theta}=\theta)$ is not necessarily uniquely defined by the joint distribution in that case.}

\begin{defi}\label{D:bayessmallProblemrv}
A \emph{discrete, small sample, Bayesian estimation problem} is a pair of random variables $(\boldsymbol{\theta},\mathbf{x})$, with a given joint probability distribution, with $\boldsymbol{\theta}$ ranging over $\Theta$ and $\mathbf{x}$ ranging over $X$, with $\supp(\boldsymbol{\theta})=\Theta$, and with $\Theta$ and $X$ possessing the same properties as in Definition \ref{D:bayessmallProblemtup}.
\end{defi}
  
While estimators are simplest to define in the small sample case, in this paper we are interested in the asymptotic behaviour of estimators in the limit of infinite data. We shall therefore be concerned with what we shall call \emph{discrete, large sample Bayesian estimation problems}:

\begin{defi}\label{D:bayesProblem}
A \emph{discrete, large sample Bayesian estimation problem} is a tuple \linebreak $\langle (X_n)_{n\in\setN}, \Theta, \mathscr{P},\Pi\rangle$ where each $X_n$ is a countable subset of some Euclidean space, and $X\defeq\prod_{n=1}^\infty X_n$ is the \emph{observation space} (in this case the infinite Cartesian product of the $X_n$), $\mathscr{P}=(P_\theta)_{\theta\in\Theta}$ is a set of distinct probability distributions over $X$ with the associated $\sigma$-algebra being the smallest $\sigma$-algebra containing all \emph{open cylinders} ($C_i[x]\defeq\{x'\in X: x'_i=x\}$, $i\in\setN$, $x\in X_i$), and $\Pi$ satisfies the same conditions as in Definition \ref{D:bayessmallProblemtup}. Alternatively a discrete, large sample Bayesian estimation problem is a pair of random variables $(\boldsymbol{\theta},\mathbf{x})$ as in Definition \ref{D:bayessmallProblemrv} with $\mathbf{x}$ ranging over an observation space which is the product of a countably infinite sequence of countable non-empty subsets of a Euclidean space, where the factors of the product are indexed by $\setN$.
\end{defi}

The intended interpretation is that the statistician will receive a sequence of observations, with the $i^{\text{th}}$ observation being a member of $X_i$. The observations form a sequence of random variables $\mathbf{x}=(\mathbf{x}_i)_{i\in\setN}$ whose distribution depends on the unknown value $\boldsymbol{\theta}$. Note that we have put no restrictions on the form of the distributions on $X$, and so in particular we do not enforce the common requirement that the observations be independent and identically distributed (i.i.d.).

All the problems we consider from this point forward will be both discrete and Bayesian, so we shall leave out these qualifiers, instead dividing problems only into small sample or large sample estimation problems. In both cases we shall use $\prob()$ when referring to probabilities defined in terms of the joint probability distribution referred to in Definitions \ref{D:bayessmallProblemrv} and \ref{D:bayesProblem}. For instance, we may write $\prob(\mathbf{x}=x|\boldsymbol{\theta}=\theta)$ instead of $P_\theta(x)$. 

Note that in all estimation problems we have required $\Theta$ and $X_i$ to be subsets of some Euclidean space. We do this solely to make our results more comprehensible due to the general familiarity with Euclidean spaces. All that is required for the results that follow is that $\Theta$ is a first countable, $T_1$ space and that all subsets of $\Theta$ and $X_i$ are measurable.

Finally, we shall need some additional notation for discussing individual observations or finite sequences of observations in large sample estimation problems:
\begin{itemize}
\item $x_{m:n}\defeq\langle x_m,x_{m+1},\ldots,x_n\rangle$. 
\item $P_\theta^{(n)}(x) \defeq \prob(\mathbf{x}_n=x|\boldsymbol{\theta}=\theta)$
\item $P_\theta^{(m:n)}(x_{m:n})\defeq\prob(\mathbf{x}_{m:n}=x_{m:n}|\boldsymbol{\theta}=\theta)$. 
\end{itemize}

\subsection{Estimators}\label{Sec:Est}
 An \emph{estimator} for a small sample estimation problem is a function $\hat{\theta}:X\rightarrow \Theta$, where $\hat{\theta}(x)$ represents a `best guess' as to the true value of $\boldsymbol{\theta}$, when given the observation $x$. Since the observation $\mathbf{x}$ is a random variable, $\hat{\theta}(\mathbf{x})$ is itself a random variable which we will denote $\randHat$. An estimator for a large sample estimation problem is a sequence of estimators $(\hat{\theta}_n)_{n\in\setN}$ for the sequence of small sample estimation problems $(T_n)_{n\in\setN}$ where $T_n=(\boldsymbol{\theta},\mathbf{x}_{1:n})$. 
 
 Estimators are thus defined relative to a given estimation problem, while we are interested in comparing methods of point estimation over whole classes of estimation problems. We thus require the notion of \emph{estimator classes} which are simply sets of estimators defined over broad classes of estimation problems. For most estimator classes we consider we will define membership explicitly only for small sample estimation problems. In such cases estimators on large sample problems will be considered members of a given estimator class if the sequence of estimators $(\hat{\theta}_n)_{n\in\setN}$ on the small sample problems $(\boldsymbol{\theta},\mathbf{x}_{1:n})_{n\in\setN}$ is eventually in the class of small sample estimators so defined. 
 
 As an example, perhaps the most well-studied class of estimators and one of the main classes we shall consider is the \emph{Maximum Likelihood (ML) estimator class}:
 
  \begin{defi}\label{D:ML}
 An estimator $\hat{\theta}$ for estimation problem $(\boldsymbol{\theta},\mathbf{x})$ is a \emph{Maximum Likelihood (ML) estimator} and so a member of \emph{the ML estimator class}  if, for all $x\in X$ it satisfies the equation:
 \begin{equation}\label{eq:ML}
 \hat{\theta}(x)=\argmax_{\theta\in\Theta}P_\theta(x).
 \end{equation}
 \end{defi}
 
 The ML estimator class is what we shall call a \emph{frequentist} estimator class, since whether an estimator on a problem $(\boldsymbol{\theta},\mathbf{x})$ is a member of the class can be determined independently of the prior distribution on $\boldsymbol{\theta}$. We shall also later introduce \emph{Bayesian} estimator classes whose membership does depend on the prior.
 
\subsection{Consistency}\label{Sec:Cons}
While our estimation problems are Bayesian, our evaluative criterion is the frequentist one of \emph{consistency}

\begin{defi}\label{D:cons1} An estimator $(\hat{\theta}_n)_{n\in\setN}$ is \emph{consistent at $\theta\in\Theta$} if, when $\mathbf{x}$ is distributed according to $P_\theta$, the sequence $(\randHat_n)_{n\in\setN}$ converges in probability to $\theta$. That is if for every neighbourhood $U$ of $\theta$,

\[\lim_{n\to\infty}\prob(\randHat_n\in U|\boldsymbol{\theta}=\theta)=1.\]

An estimator is \emph{consistent} for an estimation problem if it is consistent for all $\theta\in\Theta$.
\end{defi}

\begin{defi}\label{D:consstrong} An estimator $(\hat{\theta}_n)_{n\in\mathbb{N}}$ is \emph{strongly consistent at $\theta\in\Theta$} if, when $\mathbf{x}$ is distributed according to $P_\theta$, the sequence $(\randHat_n)_{n\in\setN}$ converges almost surely to $\theta$. That is if, 

\[\prob(\lim_{n\to\infty}\randHat=\theta|\boldsymbol{\theta}=\theta)=1.\]

If an estimator is \emph{strongly consistent} at all $\theta\in\Theta$ then it is \emph{strongly consistent}.
\end{defi}

In this paper however, we are not primarily interested in the consistency of individual estimators, but rather of estimator classes. Ideally we would define an estimator class as (strongly) consistent if every estimator within the class is (strongly) consistent. However, since our definition of estimation problems in Section \ref{Sec:EstProb} made no assumptions regarding the relationship between the sequence of observations $\mathbf{x}$ and the true value of $\boldsymbol{\theta}$, there can be no guarantee in general that a consistent estimator exists for a given estimation problem. Thus estimator classes can only guarantee consistency on  subclasses of the class of all estimation problems. We will therefore define consistency of estimator classes relative to a given class of large sample estimation problems:

\begin{defi}\label{D:classcons}
Estimator class $\mathscr{E}$ is \emph{(strongly) consistent over a class of estimation problems $\mathscr{C}$} if, for every estimation problem $T\in\mathscr{C}$, there exists at least one estimator in $\mathscr{E}$ for $T$ and all such estimators are (strongly) consistent on $T$.
\end{defi}

Note, however, that an estimator class may fail to contain estimators for a given estimation problem (for instance there may be no maximum likelihood estimator for a problem if the likelihood function for infinitely many $n$ fails to attain its supremum). We therefore also define partial consistency:

\begin{defi}\label{D:partcons}
Estimator class $\mathscr{E}$ is \emph{partially (strongly) consistent over class of estimation problems $\mathscr{C}$} if, for every estimation problem $\mathbb{P}\in\mathscr{C}$, all estimators in $\mathscr{E}$ for $\mathbb{P}$ are (strongly) consistent on $\mathbb{P}$.
\end{defi} 

\section{Proper Consistency}\label{Sec:PropCon}
In this section we introduce our broadest class of estimation problems - those with \emph{consistent posteriors}. The primary results of this section will show that a consistent posterior is a necessary condition for a problem to have consistent estimators, and that it is also a sufficient condition for an estimation problem to have strongly consistent estimators. An estimator class which is consistent over all problems with consistent posteriors we therefore call \emph{properly consistent}. We will then introduce Bayes estimators and show that Bayes estimator classes whose associated loss functions have a property we call \emph{discernment} are properly consistent. 

To conclude this section we will consider proper consistency from a frequentist perspective. We will show that the existence of a consistent posterior is equivalent to the frequentist property of having convergent likelihood ratios. This entails the existence of properly consistent frequentist estimator classes, however we will show that the Maximum Likelihood estimator class is not properly consistent.\\

\subsection{Posterior Consistency}\label{Sec:PostCons}

\begin{defi}\label{D:conspost}
\cite{ghosal1997review} An estimation problem $(\boldsymbol{\theta},\mathbf{x})$ is said to have a
\emph{consistent posterior} if, for every $\theta\in\Theta$, for every neighbourhood $U$ of $\theta$
\begin{equation}\label{Eq:conspost}
\lim_{n\to\infty} \prob(\boldsymbol{\theta}\in U|\mathbf{x}_{1:n})=1\,a.s.
\end{equation}
when $\mathbf{x}$ is distributed according to $P_\theta$. A problem without a consistent posterior is said to have an \emph{inconsistent posterior}.\footnote{Note that a probability conditioned on a random variable is itself a random variable.}
\end{defi} 

That is an estimation problem has a consistent posterior if the Bayesian posterior will, in the limit, concentrate all probability mass on regions about the true value with probability 1.  Since we only consider the discrete case in this paper, Definition \ref{D:conspost} can be simplified using L\'{e}vy's upward theorem to show that a consistent posterior will in fact concentrate all probability mass precisely on the true value.

\begin{lemma}[L\'{e}vy's upward theorem]\label{L:levy} \cite[sec. 14.2]{Williams1991} If $\mathbf{z}$ is a real random variable such that $E(|\mathbf{z}|)<\infty$, and if $(\mathbf{y}_n)_{n\in\mathbb{N}}$ is a random process and $\mathbf{y}=(\mathbf{y}_{1:\infty})$, then 
\[\lim_{n\to\infty}E(\mathbf{z}|\mathbf{y}_{1:n})=E(\mathbf{z}|\mathbf{y})\,a.s.\]
\end{lemma}

\begin{lemma}\label{L:disconspost}
For every estimation problem $T=(\boldsymbol{\theta},\mathbf{x})$, $T$ has a consistent posterior if and only if, for every $\theta\in\Theta$:
\begin{equation}\label{Eq:disconspost}
\lim_{n\to\infty} \prob(\boldsymbol{\theta}=\theta|\mathbf{x}_{1:n})=1\,a.s.
\end{equation}
when $\mathbf{x}$ is distributed according to $P_\theta$.
\end{lemma}

\begin{proof}
\emph{if}: For any neighbourhood, $U$ of $\theta$, $1\geq\prob(\boldsymbol{\theta}\in U|\mathbf{x}_{1:n})\geq\prob(\boldsymbol{\theta}=\theta|\mathbf{x}_{1:n})$. Therefore $\lim_{n\to\infty}\prob(\boldsymbol{\theta}\in U|\mathbf{x}_{1:n})=1\, a.s.$ by the order limit theorem.

\emph{only if}: Assume $(\boldsymbol{\theta},\mathbf{x})$ has a consistent posterior, and assume that $\mathbf{x}$ is distributed according to $P_\theta$. Now, for every neighbourhood, $U$, $\prob(\boldsymbol{\theta}\in U)=E(\mathbf{1}_{\{\boldsymbol{\theta}\in U\}})$ is the expectation of a random variable bounded between 0 and 1.  We can therefore  apply L\'{e}vy's upward theorem to Definition ~\ref{D:conspost} to get that for every neighbourhood, $U$, of $\theta$, $\prob(\boldsymbol{\theta}\in U|\mathbf{x})=1$ almost surely. However, by the $T_1$ property of $\setR^n$, for every $\theta'\ne\theta$ there exists a neighbourhood $U'$ of $\theta$ such that $\theta'\in U'^c$. Therefore $\prob(\boldsymbol{\theta}=\theta'|\mathbf{x})\leq\prob(\boldsymbol{\theta}\in U'^c|\mathbf{x})=0\, a.s.$ Finally since $\Theta$ is discrete:
\begin{eqnarray*}
\prob(\boldsymbol{\theta}\in U|\mathbf{x})&=&\sum_{\tilde{\theta}\in U}\prob(\boldsymbol{\theta}=\tilde{\theta}|\mathbf{x})\\
&=&\prob(\boldsymbol{\theta}=\theta|\mathbf{x})\ +\!\sum_{\tilde{\theta}\in U\backslash\{\theta\}}\prob(\boldsymbol{\theta}=\tilde{\theta}|\mathbf{x})\\
\therefore\  \prob(\boldsymbol{\theta}=\theta|\mathbf{x})&=&1\  a.s.
\end{eqnarray*}
\end{proof}

Now in order to prove the main result of this subsection (Lemma \ref{L:impcons}) we will need the following (well-known) lemma:

\begin{lemma}\label{L:subsubas}\cite[Thm. 6.3.1(b), p. 172]{Resnick1999}
A sequence of random variables $(\mathbf{y}_n)_{n\in\setN}$ converges in probability to $\mathbf{y}$ if and only if, for every subsequence $(\mathbf{y}_{n_m})_{m\in\setN}$ there exists a further subsequence $(\mathbf{y}_{n_{m_k}})_{k\in\setN}$ which converges almost surely to $\mathbf{y}$.
\end{lemma}

\begin{lemma}\label{L:impcons}
No estimator is consistent for an estimation problem with an inconsistent posterior.
\end{lemma}

\begin{proof}
Assume $(\boldsymbol{\theta},\mathbf{x})$ is an estimation problem with an inconsistent posterior. Then, applying L\'{e}vy's upward theorem to Lemma \ref{L:disconspost}, this implies that for some $\theta\in\Theta$, with positive probability  when $\mathbf{x}$ is distributed according to $P_{\theta}$,  $\prob(\boldsymbol{\theta}\ne\theta|\mathbf{x})>0$. Let $\theta^*$ be such a member of $\Theta$ and let $S\defeq\{x\in X|\prob(\boldsymbol{\theta}\ne\theta^*|\mathbf{x}=x)>0\}$ then: 

\begin{equation}\label{Eq:S>0}
 \prob(\mathbf{x}\in S|\boldsymbol{\theta}=\theta^*)>0.
 \end{equation}

Suppose $\hat{\theta}$ is a consistent estimator for the estimation problem, so that for any $\theta$, $\randHat_n(=\hat{\theta}_n(\mathbf{x}_{1:n}))$ converges to $\theta$ in probability, when $\mathbf{x}$ is drawn from $P_\theta$. Therefore, by Lemma \ref{L:subsubas}, for every subsequence $\randHat_{n_m}$ there is a further subsequence $\randHat_{n_{m_k}}$ which converges to $\theta$ almost surely when  $\mathbf{x}$ is distributed according to $P_\theta$. Thus, in particular, there exists a subsequence $\randHat_{n_{m^*}}$ such that $\prob(\randHat_{n_{m^*}}\rightarrow \theta^*|\boldsymbol{\theta}=\theta^*)=1$. Letting $A\defeq\{x\in X|\hat{\theta}_{n_{m^*}}(x_{1:n_{m^*}})\rightarrow\theta^*\}$ we can restate this as:

\begin{equation}\label{Eq:probsub}
\prob(\mathbf{x}\in A|\boldsymbol{\theta}=\theta^*)=1.
\end{equation}

Thus combining \eqref{Eq:S>0} and \eqref{Eq:probsub} we get:

\begin{equation}\label{Eq:AcapS}
\prob(\mathbf{x}\in A\cap S|\boldsymbol{\theta}=\theta^*)>0.
\end{equation}

Since $\prob(\boldsymbol{\theta}=\theta^*)>0$, this implies that $\prob(\mathbf{x}\in A\cap S)>0$. Further by the definition of S, $\prob(\boldsymbol{\theta}\ne\theta^*|\mathbf{x}\in A\cap S)>0$. So by Bayes' theorem:

\begin{equation}\label{Eq:Uc}
\prob(\mathbf{x}\in A\cap S|\boldsymbol{\theta}\ne\theta^*)>0.
\end{equation}

It follows that for some $\theta'\ne\theta^*$

\begin{equation}\label{Eq:theta'}
\prob(\mathbf{x}\in A\cap S|\boldsymbol{\theta}=\theta')>0.
\end{equation}

But again by consistency of $\hat{\theta}$ there is a subsequence $\randHat_{n_{m^*_k}}$ of $\randHat_{n_{m^*}}$ such that 
\begin{equation}\label{Eq:nmk}
\prob(\randHat_{n_{m^*_k}}\!\!\!\rightarrow\,\theta'|\boldsymbol{\theta}=\theta')=1.
\end{equation}

\eqref{Eq:theta'} and \eqref{Eq:nmk} together imply that for some $x\in A$, $\lim_{k\to\infty}\hat{\theta}_{n_{m^*_k}}(x_{1:n_{m^*_k}})=\theta'$ contradicting the definition of A. Thus $\hat{\theta}$ cannot be consistent for this problem, and therefore there can be no consistent estimator for an estimation problem with an inconsistent posterior.
\end{proof}

Lemma \ref{L:impcons} implies that any estimator class which is consistent over the class of problems with consistent posteriors is consistent on all problems for which consistent estimators exist at all. We shall call such estimator classes, \emph{properly consistent}.

\subsection{Bayes Estimators}\label{Sec:BayesEst}

We now introduce \emph{Bayes estimator classes} which form our first examples of Bayesian estimator classes. In the next subsection we will then give sufficient conditions for Bayes estimator class to be properly consistent. \emph{Bayes estimators}  \cite[p.255]{lehmann1998theory} are widely discussed examples of Bayesian estimators that attempt to minimise the posterior expectation of some \emph{loss function} on the parameter space representing the cost of guessing the model incorrectly. Specifically:

\begin{defi}\label{D:loss}
A function $L:\Theta\times\Theta\rightarrow \mathbb{R}^{\ge 0}$ is a \emph{loss function} for estimation problem $(\boldsymbol{\theta},\mathbf{x})$ if, for all $\theta\in\Theta$, $L(\theta,\theta)=0$.
\end{defi}

\begin{defi}\label{D:Bayes}
A \emph{Bayes estimator $\hat{\theta}_{L}$ associated with Loss function $L$} is any estimator satisfying:
\[\hat{\theta}_{L}(x)\in\argmin_{\theta_0} E(L(\boldsymbol{\theta},\theta_0)|\mathbf{x}=x).\]
\end{defi}

Note that while Definition \ref{D:loss} treats the two arguments of the loss function symmetrically, Definition \ref{D:Bayes} does not. The reason is that there is a semantic difference in the two arguments of the loss function. $L(\theta,\theta')$ represents the loss suffered when $\theta'$ is the estimated value and $\theta$ is the true value of $\boldsymbol{\theta}$. However we make no requirement that the loss functions themselves treat the arguments asymmetrically, and in fact many of the most popular loss functions (such as squared distance) do not. Of the loss function we introduce later in this section the discrete loss function is symmetric, while the Kullback-Leibler loss function is not.

Note further that while Definition \ref{D:Bayes} identifies all and only those estimators that are ordinarily called Bayes estimators, there are two important differences between the way we will treat Bayes estimators in this paper and the way they are treated in the literature (for instance in the complete class theorem \cite{Wald1947, Stein1955, Sacks1963}). The first difference was already noted in Section \ref{Sec:EstProb} --- namely that we treat the prior as part of the estimation problem rather than as part of the estimator.  The second difference is in the estimator classes we shall consider. Bayes estimators are typically considered individually or grouped together in a single class of all Bayes estimators (such as in the complete class theorem just mentioned). However, in order to use Bayes estimation one must first select a loss function to use. We will therefore subdivide the class of all Bayes estimators into estimator classes based on their associated loss functions.

The construction of the Bayes estimator classes thus takes a little more work, since (like estimators) loss functions only exist for a given estimation problem. To define Bayes estimator classes which exists across estimation problems with varying $\Theta$, we need the notion of a \emph{general loss function}:

\begin{defi} \label{D:genloss}
A \emph{general loss function} $\mathscr{L}$ is a function from estimation problems to loss functions. 
\end{defi}
The loss function which is the image of estimation problem $T$ under general loss function $\mathscr{L}$ we will denote with $\mathscr{L}^{(T)}$. Thus $\mathscr{L}^{(T)}(\theta,\theta')$ represents the loss from selecting $\theta'$ as an estimate when $\theta$ is the true value when applying general loss function $\mathscr{L}$ to estimation problem $T$.

In practice a loss function for a particular estimation problem is usually chosen using information from only some of the components of an estimation problem, imposing structure on the possible form of the general loss function. In particular general loss functions which are used in practice tend to either be \emph{parameter-based} or \emph{distribution-based}. A \emph{parameter-based loss function} is a general loss function such that $\mathscr{L}^{(T)}(\theta,\theta')$ depends only on the values $\theta$ and $\theta'$ (independently of $T$). A \emph{distribution-based loss function} is a general loss function such that $\mathscr{L}^{(T)}(\theta,\theta')$ depends only on the distributions $P_\theta$ and $P_{\theta'}$ in $T$. For instance the general loss function which puts the discrete metric on every space: 
\[\mathscr{L}_\text{disc}^{(T)}(\theta,\theta')\defeq \left\{\begin{array}{ll}
1 & \text{if }\theta\ne\theta'\\
0 & \text{if }\theta=\theta'
\end{array}
\right.
\]
is a parameter-based loss function. An example of a popular distribution-based loss function is the following, based on Kullback-Leibler divergence:
\[\mathscr{L}_\text{KL}^{(T)}(\theta,\theta')\defeq D_{\text{KL}}(P_\theta,P_{\theta'})\defeq E_{P_\theta}\left(\log\RN{P_\theta}{P_{\theta'}}\right).\]

The Bayes estimator class associated with a given general loss function $\mathscr{L}$ is the class of Bayes estimators whose associated loss function is the image of the estimation problem the estimator is defined on, under the general loss function. A Bayes estimator class whose associated general loss function is parameter-based (distribution-based) we will call a \emph{parameter-based (distribution-based) Bayes estimator class}. Thus the class of Bayes estimators minimising the discrete metric is a parameter-based Bayes estimator class which is more commonly known as the \emph{Maximum A Posteriori (MAP) estimator class}, and the class of Bayes estimators minimising Kullback-Leibler divergence is a distribution-based Bayes estimator class which we will follow \cite{DoweBaxterOliverWallace1998} in calling the \emph{minimum Expected Kullback-Leibler (minEKL) estimator class}. If an individual Bayes estimator is a member of a parameter-based (distribution-based) Bayes estimator class we will call it a \emph{parameter-based (distribution-based) estimator}. This implies that an estimator may be both a parameter-based and a distribution-based Bayes estimator, while an estimator \emph{class} cannot be both parameter-based and distribution-based.

We have dealt explicitly only with small sample estimation problems in defining Bayes estimator classes, which in light of Section \ref{Sec:Est} is sufficient to define Bayes estimators and estimator classes over large sample estimation problems as well. There is an important point to note about the associated loss functions in large sample estimation problems however. In the large sample case the distributions associated with each member of $\Theta$, and so the associated loss functions of distribution-based Bayes estimators, depends on the number of data points observed. Thus a distribution-based Bayes estimator on a large sample estimation does not have a single associated loss function. Instead such an estimator has a sequence of associated loss functions $(L_n)_{n\in\setN}$. While the associated loss function of a parameter-based Bayes estimator does not depend on the number of data points observed, in order to provide a uniform treatment of the two types of Bayes estimator classes we shall consider all Bayes estimators on large sample estimation problems to have an associated \emph{loss function sequence} $(L_n)_{n\in\setN}=(\mathscr{L}^{(T_n)})_{n\in\setN}$ instead of an associated loss function (with parameter-based Bayes estimators' associated loss function sequences being constant).

\subsection{Consistency of Bayes Estimator Classes}

One natural way to show that a Bayes estimator is consistent is to show that: 
\begin{enumerate}
\item The expected loss of the true value converges to zero, while 
\item The expected loss of values far away from the true value are eventually bounded away from zero.
\end{enumerate}
With $\boldsymbol{\theta}=\theta$, this can fail if either the loss $L_n(\theta,\theta')$ can get too small so there is some sequence of estimates not converging to $\theta$ which has expected losses converging to zero, or conversely if $L_n(\theta',\theta)$ can get too large so that $E(L_n(\boldsymbol{\theta},\theta)|\mathbf{x}_{1:n})$ does not converge to zero. This provides the intuition for the following definition of a \emph{discerning} sequence of loss functions, which is designed to avoid both of these issues. 

\begin{defi}\label{D:seqdisc}		
The sequence of loss functions $(L_n)_{n\in\setN}$ is \emph{discerning} if for all $\theta\in\Theta$, and any neighbourhood $U$ of $\theta$, 
\begin{equation}\label{eq:disc}
\liminf_{n\to\infty} \inf_{\theta'\in U^c} \frac{L_n(\theta,\theta')}{K^{(\theta)}_n}>0
\end{equation}
where $K^{(\theta)}_n\defeq \sup_{\theta'\in\Theta}L_n(\theta',\theta)$. Loss function $L$ is \emph{discerning} if the constant sequence $(L)_{n\in\setN}$ (i.e., the sequence $(L_n)_{n\in\setN}$ where for all $n$, $L_n=L$) is discerning. A general loss function sequence $\mathscr{L}$ is \emph{discerning} if, for every large sample estimation problem $T$, the loss function sequence $(\mathscr{L}^{(T_n)})_{n\in\setN}$ is discerning.
\end{defi}

Note that for constant loss function sequences $(L)_{n\in\setN}$ equation \eqref{eq:disc} reduces to 
\[\frac{\inf_{\theta'\in U^c}L(\theta,\theta')}{K^{(\theta)}}>0,\]
where $K^{(\theta)}\defeq\sup_{\theta'\in\Theta}L(\theta',\theta)$. This holds if only if $\inf_{\theta'\in U^c}L(\theta,\theta')>0$ and $K^{(\theta)}<\infty$.

Having a discerning associated loss function sequence is a sufficient condition for the consistency of a Bayes estimator:
 
\begin{lemma}\label{L:disccons}
If $T=(\boldsymbol{\theta},\mathbf{x})$ is a large sample estimation problem with a consistent posterior, and $(L_n)_{n\in\setN}$ is a discerning loss function sequence, then any Bayes estimator on $T$ associated with $(L_n)$, $\hat{\theta}_{L_n}$, is strongly consistent on $T$.
\end{lemma}

\begin{proof}
Let $T=(\boldsymbol{\theta},\mathbf{x})$ be a large sample estimation problem with a consistent posterior, let $(L_n)_{n\in\setN}$ be a discerning loss function sequence, and for each $\theta\in\Theta$, let $A_\theta\defeq\{x\in X|\lim_{n\to\infty}\prob(\boldsymbol{\theta}=\theta|\mathbf{x}_{1:n}=x_{1:n})=1\}$. Then Lemma \ref{L:disconspost} implies:
\begin{equation}\label{eq:Aeq1}
\prob(\mathbf{x}\in A_\theta|\boldsymbol{\theta}=\theta)=1.
\end{equation}

Now suppose that $(\hat{\theta}_{L_n})_{n\in\setN}$ is not strongly consistent. Then there exists a $\theta^*\in\Theta$ such that $(\boldsymbol{\hat{\theta}}_{L_n})_{n\in\setN}$ does not converge almost surely to $\theta^*$ when $\boldsymbol{\theta}=\theta^*$. Thus let  $B_\theta\defeq\{x\in X|\hat{\theta}_n(x_{1:n})\nrightarrow \theta\}$, so that by assumption 
\begin{equation}\label{eq:Bgeq0}
\prob(\mathbf{x}\in B_{\theta^*}|\boldsymbol{\theta}=\theta^*)>0
\end{equation}
Then from \eqref{eq:Aeq1} and \eqref{eq:Bgeq0} we get:
\[
\prob(\mathbf{x}\in A_{\theta^*}\cap B_{\theta^*}|\boldsymbol{\theta}=\theta^*)>0,
\]
and
\[
\prob(\mathbf{x}\in A_{\theta^*}|\text{$\mathbf{x}\in B_{\theta^*}$ and $\boldsymbol{\theta}=\theta^*$})=1.
\]

Therefore $A_{\theta^*}\cap B_{\theta^*}$ is not empty so we may choose $x^*\in  A_{\theta^*}\cap B_{\theta^*}$. Now note that since $(L_n)_{n\in\setN}$ is discerning, for all $\theta\in \Theta$, $K_n^{(\theta)}<\infty$. Further, because by definition $K_n^{(\theta)}=\sup_{\theta'\in\Theta}L_n(\theta',\theta)$, we have that for any $\theta'$, $\frac{L_n(\theta',\theta)}{K^{(\theta)}_n}\leq 1$, and because $L_n$ is a loss function, $\frac{L_n(\theta,\theta)}{K^{(\theta)}_n}=0$. Combining these two facts, we get:

\begin{equation}
E\left(\frac{L_n(\boldsymbol{\theta},\theta)}{K^{(\theta)}_n}\middle|\mathbf{x}_{1:n}=x_{1:n}\right)\leq 1-\prob\left(\boldsymbol{\theta}=\theta\middle|\mathbf{x}_{1:n}=x_{1:n}\right).
\end{equation}
Therefore, since $\lim_{n\to\infty}\prob(\boldsymbol{\theta}=\theta^*|\mathbf{x}_{1:n}=x^*_{1:n})=1$:

\begin{equation}\label{eq:LKto0}
\lim_{n\to\infty}E\left(\frac{L_n(\boldsymbol{\theta},\theta^*)}{K^{(\theta^*)}_n}\middle|\mathbf{x}_{1:n}=x^*_{1:n}\right)=0.
\end{equation}

Now, $E\left(\frac{L_n(\boldsymbol{\theta},\boldsymbol{\hat{\theta}}_{L_n})}{K^{(\theta^*)}_n}\middle|\mathbf{x}_{1:n}=x^*_{1:n}\right)=\sum_{\theta'\in\Theta}\prob(\boldsymbol{\theta}=\theta'|\mathbf{x}_{1:n}=x^*_{1:n})\frac{L_n(\theta',\hat{\theta}_{L_n}(x^*_{1:n}))}{K^{(\theta^*)}_n}$, so
\begin{equation}
E\left(\frac{L_n(\boldsymbol{\theta},\boldsymbol{\hat{\theta}}_{L_n})}{K^{(\theta^*)}_n}\middle|\mathbf{x}_{1:n}=x^*_{1:n}\right)\geq\prob(\boldsymbol{\theta}=\theta^*|\mathbf{x}_{1:n}=x^*_{1:n})\frac{L_n(\theta^*,\hat{\theta}_{L_n}(x^*_{1:n}))}{K_n^{(\theta^*)}}.\label{eq:geqprob}
\end{equation}
Since $x^*\in A_{\theta^*}$, $\lim_{n\to\infty}\prob(\boldsymbol{\theta}=\theta^*|\mathbf{x}_{1:n}=x^*_{1:n})=1$, so \eqref{eq:geqprob} implies:

\begin{equation}\label{eq:Etolimsup}
\limsup_{n\to\infty} E\left(\frac{L_n(\boldsymbol{\theta},\boldsymbol{\hat{\theta}}_{L_n})}{K^{(\theta^*)}_n}\middle|\mathbf{x}_{1:n}=x^*_{1:n}\right)\geq \limsup_{n\to\infty}\frac{L_n(\theta^*,\hat{\theta}_{L_n}(x^*_{1:n}))}{K^{(\theta^*)}_n}.
\end{equation}

However, since $x^*\in B_{\theta^*}$, there exists a neighbourhood $U$ of $\theta^*$, and a subsequence $(\hat{\theta}_{L_{n_k}}(x^*_{1:n_k}))_{k\in\mathbb{N}}$ of $(\hat{\theta}_{L_n}(x^*_{1:n}))_{n\in\mathbb{N}}$, such that $(\hat{\theta}_{L_{n_k}}(x^*_{1:n_k}))_{k\in\mathbb{N}}$ is in $U^c$. Hence

\begin{equation}\label{eq:limsuptoliminf}
\limsup_{n\to\infty}\frac{L_n(\theta^*,\hat{\theta}(x^*_{1:n}))}{K^{(\theta^*)}_n}\geq\liminf_{k\to\infty} \frac{L_{n_k}(\theta^*,\hat{\theta}_{L_{n_k}}(x^*_{1:n_k}))}{K^{(\theta^*)}_{n_k}}\geq\liminf_{k\to\infty}\inf_{\theta'\notin U} \frac{L_{n_k}(\theta^*,\theta')}{K^{(\theta^*)}_{n_k}}.
\end{equation}

But since the liminf of a subsequence is at least as great as the liminf of the original sequence \eqref{eq:limsuptoliminf} implies:

\begin{equation}\label{eq:liminf>0}
\limsup_{n\to\infty}\frac{L_n(\theta^*,\hat{\theta}(x^*_{1:n}))}{K^{(\theta^*)}_n}\geq\liminf_{n\to\infty}\inf_{\theta'\in U^c} \frac{L_n(\theta^*,\theta')}{K^{(\theta^*)}_n}>0,
\end{equation}
where the final inequality follows from $(L_n)_{n\in\setN}$ being discerning. Combining \eqref{eq:Etolimsup} and \eqref{eq:liminf>0} we conclude $\limsup_{n\to\infty} E\left(\frac{L_n(\boldsymbol{\theta},\boldsymbol{\hat{\theta}}_{L_n})}{K^{(\theta^*)}_n}\middle|\mathbf{x}_{1:n}=x^*_{1:n}\right)>0$. Therefore there exists a $c>0$ such that for infinitely many $n\in\setN$,  $E\left(\frac{L_n(\boldsymbol{\theta},\boldsymbol{\hat{\theta}}_{L_n})}{K^{(\theta^*)}_n}\middle|\mathbf{x}_{1:n}=x^*_{1:n}\right)>c$.   In contrast, it follows from \eqref{eq:LKto0} that $\lim_{n\to\infty}E\left(\frac{L_n(\boldsymbol{\theta},\theta^*)}{K^{(\theta^*)}_n}\middle|\mathbf{x}_{1:n}=x^*_{1:n}\right)=0$, so there exists an $N\in\setN$ such that for all $n>N$, $E\left(\frac{L_n(\boldsymbol{\theta},\theta^*)}{K^{(\theta^*)}_n}\middle|\mathbf{x}_{1:n}=x^*_{1:n}\right)<c$. Therefore, for infinitely many $k>N$:

\begin{equation}\label{eq:somek}
E(L_k(\boldsymbol{\theta},\theta^*)|\mathbf{x}_{1:k}=x^*_{1:k})<E(L_k(\boldsymbol{\theta},\boldsymbol{\hat{\theta}}_{L_k})|\mathbf{x}_{1:k}=x^*_{1:k})
\end{equation}

Since $E(L_k(\boldsymbol{\theta},\boldsymbol{\hat{\theta}}_{L_k})|\mathbf{x}_{1:k}=x^*_{1:k})=E(L_k(\boldsymbol{\theta},\hat{\theta}_{L_k}(x^*_{1:k}))|\mathbf{x}_{1:k}=x^*_{1:k})$, equation \eqref{eq:somek} implies:

\begin{equation*}
\hat{\theta}_{L_k}(x^*_{1:k})\ne\argmin_{\theta'}E(L_k(\boldsymbol{\theta},\theta')|\mathbf{x}_{1:k}=x^*_{1:k}),
\end{equation*}

contradicting the definition of $\hat{\theta}_{L_k}$ (Definition \ref{D:Bayes}). Hence, our assumption that $(\hat{\theta}_{L_n})_{n\in\setN}$ is not strongly consistent must be false, proving the theorem.
\end{proof}

Lemma \ref{L:disccons} thus defines a family of estimator classes which are (strongly) properly consistent. It remains to exhibit a member of the family, to show that it is non-empty and give our first example of a (strongly) properly consistent estimator class. Our next Lemma shows that the class of MAP estimators is such a class, and is therefore strongly properly consistent.

\begin{cor}\label{C:MAPcons}
The MAP estimator class is strongly properly consistent.
\end{cor}

\begin{proof}
The discrete metric is discerning since for any $\theta$ and $U$, $K^{(\theta)}=1=\inf_{\theta'\in U^c} L(\theta,\theta')$. Thus by Lemma \ref{L:disccons} every MAP estimator on a problem with a consistent posterior is strongly consistent. That is the MAP estimator class is strongly partially consistent over problems with consistent posteriors.

 In order to show that MAP is strongly properly consistent, we must additionally show that for every problem $T$ with a consistent posterior, there exists a MAP estimator for $T$: Suppose for contradiction that no MAP estimator exists for $T$. Then for some $n\in\mathbb{N}$ and $x\in X$, $\argmax_{\theta\in\Theta}\prob(\boldsymbol{\theta}=\theta|\mathbf{x}_{1:n}=x_{1:n})$ does not exist. In other words for all $\theta\in\Theta$, $\prob(\boldsymbol{\theta}=\theta|\mathbf{x}_{1:n}=x_{1:n})<\sup_{\theta'\in\Theta}\prob(\boldsymbol{\theta}=\theta'|\mathbf{x}_{1:n}=x_{1:n})$.  It follows that $\sup_{\theta'\in\Theta}\prob(\boldsymbol{\theta}=\theta'|\mathbf{x}_{1:n}=x_{1:n})>0$, and that for all $\epsilon>0$ there exists infinitely many $\theta\in\Theta$ such that $\prob(\boldsymbol{\theta}=\theta|\mathbf{x}_{1:n}=x_{1:n})>\sup_{\theta'\in\Theta}\prob(\boldsymbol{\theta}=\theta'|\mathbf{x}_{1:n}=x_{1:n})-\epsilon$. This then implies $\sum_{\theta\in\Theta}\prob(\boldsymbol{\theta}=\theta|\mathbf{x}_{1:n}=x_{1:n})=\infty$, but we know that $\sum_{\theta\in\Theta}\prob(\boldsymbol{\theta}=\theta|\mathbf{x}_{1:n}=x_{1:n})=1$ by the laws of conditional probability. Hence, the assumption that no MAP estimator for $T$ exists must be false, and therefore the class of MAP estimators is strongly properly consistent.
\end{proof}

The results concerning posterior consistency can now be neatly summarised in the following main result:

\begin{thm}\label{T:PropCons}
For any estimation problem $T$ the following conditions are all equivalent:
\begin{enumerate}
\item $T$ has a consistent posterior.
\item There exists a consistent estimator for $T$.
\item There exists a strongly consistent estimator for $T$.
\end{enumerate}
\end{thm}

\begin{proof}
$3\implies 2$: Every estimator which is strongly consistent is consistent since almost sure convergence implies convergence in probability.\\
$2\implies 1$: Lemma \ref{L:impcons}\\
$1\implies 3$: Corollary \ref{C:MAPcons} gives an example of a strongly consistent estimator for any problem with a consistent posterior.
\end{proof}  

While we have seen that the MAP estimator class is strongly properly consistent, this does not hold for Bayes estimator classes in general, and, in fact, no distribution-based Bayes estimator class is properly consistent.

\begin{thm}
No distribution-based Bayes estimator class is properly consistent.
\end{thm}

\begin{proof}
Consider the large sample estimation problem $T=(\mathbf{x},\boldsymbol{\theta})$, where $\Theta=\setN\cup\{0\}$, $\prob(\boldsymbol{\theta}=\theta)=2^{-1-\theta}$, $X_n=\{0,1\}$ for all $n\in\setN$,  and the distribution of $\mathbf{x}$ is given by:

\begin{equation*}
\prob(\mathbf{x}_n=1|\boldsymbol{\theta}=\theta)=\left\{\begin{array}{ll}
1 & \text{if }\theta=0\text{ or }n\leq \theta\\
0 & \text{otherwise.}
\end{array}
\right.
\end{equation*}

Thus $\mathbf{x}$ is deterministic when conditioned on $\boldsymbol{\theta}=\theta$, and takes the value of a sequence of $\theta$ ones followed by an infinite sequence of zeros, or, if $\theta=0$, an infinite sequence of ones. Note that since the MAP estimator is consistent on $T$, $T$ has a consistent posterior.

Now consider the estimator $(\hat{\theta}_n)_{n\in\setN}$ for this problem where $\hat{\theta}_n(1^n0^m)=n$. For any sequence of observations, and any distribution-based loss function $\hat{\theta}$ has a posterior expected loss of 0, and so is a member of every distribution-based estimator class. However $\prob(\boldsymbol{\hat{\theta}}_n\rightarrow \infty|\boldsymbol{\theta}=0)=1$, so $\hat{\theta}$ is inconsistent. Therefore, no distribution-based Bayes estimator class is properly consistent.
\end{proof}

\subsection{Frequentist Proper Consistency}
Up until this point, our approach to properly consistent estimator classes has been entirely Bayesian: proper consistency was defined using the Bayesian posterior and we have thus far only considered the requirements for Bayes estimators to be properly consistent. In this subsection we consider a frequentist approach and show that Bayesianism is unnecessary for both defining proper consistency and for attaining properly consistent estimator classes. In particular, we will show that the consistent posterior property is equivalent to the frequentist property of having \emph{convergent likelihood ratios}. 

The equivalence between posterior consistency and convergent likelihood ratios implies that for any properly consistent Bayesian estimator class  $\mathscr{C}$, a frequentist estimator class $\mathscr{C}^{(F)}$ can be constructed by simply choosing a canonical prior $\Pi^{(F)}_\Theta$ for each possible $\Theta$, and then declaring an estimator $\hat{\theta}$ on estimation problem $T=\langle X, \Theta, \mathscr{P},\Pi\rangle$ to be a member of $\mathscr{C}^{(F)}$ if and only if the same estimator on $ T^{(F)}=\langle X, \Theta, \mathscr{P},\Pi^{(F)}_\Theta \rangle$ is a member of $\mathscr{C}$. 

While it is thus possible to form a frequentist properly consistent estimator class, the most popular frequentist estimator class, Maximum Likelihood, is not properly consistent. There are famous examples of the inconsistency of Maximum Likelihood which can be used to show failure of proper consistency. (For instance, in \cite{Hannan1960} a discrete problem is given, where Maximum Likelihood fails to be consistent, despite the existence of consistent estimators for the problem.) In the conclusion of this section, however, we show that the Maximum Likelihood estimator class is not properly consistent by means of an original example, which we introduce because it provides the motivation for our notion of \emph{distinctive likelihoods} to be introduced in the next section.

\begin{defi}\label{D:convratio} A large sample estimation problem has \emph{convergent likelihood ratios} if, for all $\theta,\theta'\in\Theta$ such that $\theta'\ne\theta$,  the likelihood ratio of $\theta'$ over $\theta$ converges almost surely to 0 conditioned on $\theta$ being the true model. That is:
\[\prob\left(\lim_{n\to\infty}\frac{P_{\theta'}^{(1:n)}(\mathbf{x}_{1:n})}{P_{\theta}^{(1:n)}(\mathbf{x}_{1:n})}=0\middle|\boldsymbol{\theta}=\theta\right)=1\]
\end{defi}

Note that, unlike Definition \ref{D:conspost}, Definition \ref{D:convratio} makes no use of the prior distribution on $\boldsymbol{\theta}$.

\begin{lemma}\label{L:freqcons}
A large sample estimation problem has a consistent posterior if and only if it has convergent likelihood ratios.
\end{lemma}

\begin{proof}
\emph{only if}: Suppose first that $T=(\boldsymbol{\theta},\mathbf{x})$ is a large sample estimation problem with a consistent posterior. Then when $\mathbf{x}$ is distributed according $P_\theta$, $\prob(\boldsymbol{\theta}=\theta|\mathbf{x})=1$ almost surely by Lemma ~\ref{L:disconspost} and so $\prob(\boldsymbol{\theta}=\theta'|\mathbf{x})=0$ almost surely for any $\theta'\ne\theta$. Thus, taking a fixed $\theta'\ne\theta$, by L\'{e}vy's upward theorem both of the following hold  with probability 1: $\lim_{n\to\infty}\prob(\boldsymbol{\theta}=\theta|\mathbf{x}_{1:n})=1$ and $\lim_{n\to\infty}\prob(\boldsymbol{\theta}=\theta'|\mathbf{x}_{1:n})=0$. Therefore $\lim_{n\to\infty} \left(\frac{\prob(\boldsymbol{\theta}=\theta'|\mathbf{x}_{1:n})}{\prob(\boldsymbol{\theta}=\theta|\mathbf{x}_{1:n})}\right)=0$ almost surely.

Now, let $S_n$ be the support of $\mathbf{x}_{1:n}$, so that for all $x\in S$ we get $\frac{\prob(\boldsymbol{\theta}=\theta'|\mathbf{x}_{1:n}=x_{1:n})}{\prob(\boldsymbol{\theta}=\theta|\mathbf{x}_{1:n}=x_{1:n})}= \frac{\prob(\boldsymbol{\theta}=\theta')}{\prob(\boldsymbol{\theta}=\theta)}\frac{P_{\theta'}^{(1:n)}(x_{1:n})}{P_{\theta}^{(1:n)}(x_{1:n})}$ by applying Bayes' theorem to the numerator and denominator. Then since $\prob(\mathbf{x}_{1:n}\in S)=1$, $\frac{\prob(\boldsymbol{\theta}=\theta'|\mathbf{x}_{1:n})}{\prob(\boldsymbol{\theta}=\theta|\mathbf{x}_{1:n})}=\frac{\prob(\boldsymbol{\theta}=\theta')}{\prob(\boldsymbol{\theta}=\theta)}\frac{P_{\theta'}^{(1:n)}(\mathbf{x}_{1:n})}{P_{\theta}^{(1:n)}(\mathbf{x}_{1:n})}$ almost surely. Therefore, when $\mathbf{x}$ is distributed according to $P_\theta$, $\lim_{n\to\infty} \frac{\prob(\boldsymbol{\theta}=\theta')}{\prob(\boldsymbol{\theta}=\theta)}\frac{P_{\theta'}^{(1:n)}(\mathbf{x}_{1:n})}{P_{\theta}^{(1:n)}(\mathbf{x}_{1:n})}=0$ almost surely. Since $0<\frac{\prob(\boldsymbol{\theta}=\theta')}{\prob(\boldsymbol{\theta}=\theta)}<\infty$ and this value does not depend on $n$, by the algebraic limit theorem \cite[Thm 2.3.3]{Abbott2000} $\lim_{n\to\infty}\frac{P_{\theta'}^{(1:n)}(\mathbf{x}_{1:n})}{P_{\theta}^{(1:n)}(\mathbf{x}_{1:n})}=0\,a.s.$, proving the first half of the claim.

\emph{if}: Suppose for any $\theta,\theta'\in\Theta$ such that $\theta'\ne\theta$, $\prob\left(\lim_{n\to\infty}\frac{P_{\theta'}^{(1:n)}(\mathbf{x}_{1:n})}{P_{\theta}^{(1:n)}(\mathbf{x}_{1:n})}=0\middle|\boldsymbol{\theta}=\theta\right)=1$. Applying Bayes' theorem this implies that when $\mathbf{x}$ is distributed according to $P_\theta$

\begin{equation}\label{Eqn:flippedRN}
\lim_{n\to\infty}\frac{\prob(\boldsymbol{\theta}=\theta)\prob(\boldsymbol{\theta}=\theta'|\mathbf{x}_{1:n})}{\prob(\boldsymbol{\theta}=\theta')\prob(\boldsymbol{\theta}=\theta|\mathbf{x}_{1:n})}=0\,a.s..
\end{equation}

Since $0<\frac{\prob(\boldsymbol{\theta}=\theta)}{\prob(\boldsymbol{\theta}=\theta')}<\infty$ and this value does not depend on $n$, equation \eqref{Eqn:flippedRN} implies that $\lim_{n\to\infty}\frac{\prob(\boldsymbol{\theta}=\theta'|\mathbf{x}_{1:n})}{\prob(\boldsymbol{\theta}=\theta|\mathbf{x}_{1:n})}=0$ almost surely.

By the algebraic limit theorem, we conclude that $\lim_{n\to\infty}\prob(\boldsymbol{\theta}=\theta'|\mathbf{x}_{1:n})=0$ and so by L\'{e}vy's upward theorem, $\prob(\boldsymbol{\theta}=\theta'|\mathbf{x})=0$ for all $\theta'\ne \theta$.
Since this holds for all $\theta'\ne\theta$ and $\sum_{\tilde{\theta}\in\Theta}\prob(\boldsymbol{\theta}=\tilde{\theta}|\mathbf{x})=1$ because $\Theta$ is discrete, $\prob(\boldsymbol{\theta}=\theta|\mathbf{x})=1$ almost surely. Hence, by L\'{e}vy's upward theorem, when $\mathbf{x}$ is distributed according to $P_\theta$, $\lim_{n\to\infty}\prob(\boldsymbol{\theta}=\theta|\mathbf{x}_{1:n})=1$ almost surely. Thus, the posterior is consistent by Lemma \ref{L:disconspost}.
\end{proof}

\begin{cor}\label{C:feqprop}
If any of the conditions in Theorem \ref{T:PropCons} are satisfied by estimation problem $ T=\langle X, \Theta, \mathscr{P},\Pi\rangle$, then they are also satisfied for any other estimation problem $T'=\langle X, \Theta, \mathscr{P},\Pi'\rangle$, identical to $T$ except for the prior distribution.
\end{cor}

\begin{proof}
By Lemma \ref{L:freqcons} if $T=\langle X, \Theta, \mathscr{P},\Pi\rangle$ satisfies any of the conditions of Theorem \ref{T:PropCons} then is has convergent likelihood ratios. But since whether a problem has convergent likelihood ratios is independent of the problem's prior $T'=\langle X, \Theta, \mathscr{P},\Pi'\rangle$ also has convergent likelihood ratios, and so by Lemma \ref{L:freqcons} satisfies all of the conditions of Theorem \ref{T:PropCons}.
\end{proof}

As an example of convergent likelihood ratios, consider the
large sample estimation problem of estimating
the probability parameter of a sequence of
i.i.d., Bernoulli-distributed random variables. This
will form a running example, so it is worth defining explicitly.

\begin{defi}
Let $p:\Theta\to(0,1)$ be a one-to-one function.

A \emph{binomial probability estimation problem} is a large sample estimation
problem where for each $\theta\in\Theta$, each $\mathbf{x}_i$ is, given
$\boldsymbol{\theta}=\theta$, independently Bernoulli-distributed with
probability parameter $p(\theta)$.
\end{defi}

\begin{lemma}\label{L:Bin_CLR}
Every binomial probability estimation problem has convergent likelihood ratios.
\end{lemma}

\begin{proof}
A Bernoulli distribution admits two possible values. Under the assumption
$\boldsymbol{\theta}=\theta$, one of these values has probability
$p(\theta)$ and the other $1-p(\theta)$. Let $k(n)$ be the number of
observations, among $\mathbf{x}_{1:n}$, to attain the value of
probability $p(\theta)$, and let $\rho(n)=k(n)/n$.

The strong law of large numbers states that $\rho(n)$ converges almost
surely to $p(\theta)$.

Consider, now, that for any $\theta'\in\Theta$,
\begin{equation}\label{Eq:BinL}
P^{(1:n)}_{\theta'}(\mathbf{x}_{1:n})=p(\theta')^{k(n)}(1-p(\theta'))^{n-k(n)}
=\left(p(\theta')^{\rho(n)}(1-p(\theta'))^{1-\rho(n)}\right)^n,
\end{equation}
and that as $\rho(n)$ converges to $p(\theta)$, the value of
$p(\theta')^{\rho(n)}(1-p(\theta'))^{1-\rho(n)}$ converges to
$p(\theta')^{p(\theta)}(1-p(\theta'))^{1-p(\theta)}$.

The function $f(t)=t^{p(\theta)}(1-t)^{1-p(\theta)}$ has a unique maximum
at $t=p(\theta)$. Hence, with probability $1$ given
$\boldsymbol{\theta}=\theta$ and $\theta'\ne\theta$,
\begin{equation}\label{Eq:Bin_ratio_lim}
\lim_{n\to\infty} \frac{p(\theta')^{\rho(n)}(1-p(\theta'))^{1-\rho(n)}}{p(\theta)^{\rho(n)}(1-p(\theta))^{1-\rho(n)}}=\frac{f(p(\theta'))}{f(p(\theta))}<1,
\end{equation}
so combining \eqref{Eq:BinL} and \eqref{Eq:Bin_ratio_lim}, we get
\[
\lim_{n\to\infty} \frac{P^{(1:n)}_{\theta'}(\mathbf{x}_{1:n})}{P^{(1:n)}_{\theta}(\mathbf{x}_{1:n})}=0.
\]
\end{proof}

Finally we provide our example of maximum likelihood inconsistency.

\begin{thm}\label{T:MLimp}
Maximum Likelihood is not properly consistent
\end{thm}

\begin{proof}
For $\theta$ in the range $0<\theta<1$, let $([\theta]^b_i)_{i\in\setN}$ denote the base
$b$ breakdown of $\theta$, as follows:
\[
\theta=\sum_{i=1}^{\infty} [\theta]^b_i b^{-i},
\]
where for all $i$,  $[\theta]^b_i\in\{0,\ldots,b-1\}$. Then let $S(\theta)=\{i\in\setN:[\theta]^3_i=1\}$ and $S_i=\{\theta\in(0,1):S(\theta)=\{i+1,i+2,\ldots\}\}$. Note that the $S_i$ are disjoint so we can define the function $s$ over their union such that $s(\theta)=i$ if and only if $\theta\in S_i$. 

Now, let $T=(\boldsymbol{\theta},\mathbf{x})$ be a large sample estimation problem where $\Theta=\bigcup_{i=1}^\infty S_i$, $\prob(\boldsymbol{\theta}=\theta)=7\cdot 16^{-s(\theta)}$ and for all $i\in\setN$, $X_i  = \{0,1\}$ with the $\mathbf{x}_i$ independent when conditioned on $\boldsymbol{\theta}$, with conditional distributions given by:
\[
\prob(\mathbf{x}_i=1|\boldsymbol{\theta}=\theta)=\begin{cases}
[\theta]^3_i/2 & \text{if $i\le s(\theta)$} \\
\theta & \text{otherwise}.
\end{cases}
\]

We will prove that ML is inconsistent for this problem, even though it has convergent likelihood ratios (and therefore a consistent posterior).
First, to see that ML is inconsistent, note that for any $\theta\in\Theta$, if $i>s(\theta)$  then after $i$ observations the likelihood of $\theta$ will be smaller than $1$. However, if $\mathbf{x}=x$ then for any $\theta'$ value
for which $\forall j\le i$,  $[\theta']^3_j/2=x_j$ will have the larger likelihood
of $1$. Thus, $\theta$ is not the maximum likelihood solution. As $i$ grows
to infinity, the ML estimate will converge to the value $\sum_{i=1}^{\infty} 2x_i 3^{-i}$. This value is clearly
distinct from $\theta$, differing from it by at least $3^{-s(\theta)}/6$.

To see that the problem has convergent likelihood ratios (and therefore a consistent posterior) consider any $\theta'\in\Theta$ such that $\theta'\ne\theta$, and let $A=\max\{s(\theta),s(\theta')\}$. Then:

\[
\frac{P_{\theta'}^{(1:A+n)}(\mathbf{x}_{1:A+n})}{P_\theta^{(1:A+n)}(\mathbf{x}_{1:A+n})}=\frac{P_{\theta'}^{(1:A)}(\mathbf{x}_{1:A})}{P_\theta^{(1:A)}(\mathbf{x}_{1:A})}\cdot\frac{P_{\theta'}^{(A+1:A+n)}(\mathbf{x}_{A+1:A+n})}{P^{(A+1:A+n)}_\theta(\mathbf{x}_{A+1:A+n})}.
\]
However, $P^{(A+1:A+n)}_\theta$ defines the distribution of each
$\mathbf{x}_i$ in $i\in [A+1,A+n]$ given $\theta$ as an independent
Bernoulli distribution, with probability parameter $p(\theta)=\theta$.
By Lemma~\ref{L:Bin_CLR}, this problem has convergent likelihood ratios,
so with probability $1$ given $\boldsymbol{\theta}=\theta$ and
$\theta'\ne\theta$, we have that
\[
\lim_{n\to\infty}\frac{P_{\theta'}^{(A+1:A+n)}(\mathbf{x}_{1:n})}{P_\theta^{(1:n)}(\mathbf{x}_{1:n})}=0,
\]
and therefore also
\[
\lim_{n\to\infty}\frac{P_{\theta'}^{(1:n)}(\mathbf{x}_{1:n})}{P_\theta^{(1:n)}(\mathbf{x}_{1:n})}=\frac{P_{\theta'}^{(1:A)}(\mathbf{x}_{1:A})}{P_\theta^{(1:A)}(\mathbf{x}_{1:A})}\cdot\lim_{n\to\infty}\frac{P_{\theta'}^{(A+1:A+n)}(\mathbf{x}_{1:n})}{P_\theta^{(1:n)}(\mathbf{x}_{1:n})}=0.
\]
\end{proof}

\section{Likelihood Consistency}\label{S:LC}

The issue that the example in Theorem \ref{T:MLimp}  creates for the consistency of  Maximum Likelihood is straightforward, namely that although the likelihood ratios converge to 0 with the true model likelihood in the denominator, they do not do so uniformly. Therefore at no point are all the likelihood ratios less than 1. Requiring the ratios to converge uniformly gives the class of estimation problems with \emph{distinctive likelihoods}, and estimator classes which are consistent over the class of estimation problems with distinctive likelihoods we call \emph{likelihood consistent}. We show that having distinctive likelihoods is sufficient to ensure the consistency of all maximum likelihood estimators. Unfortunately maximum likelihood estimators do not exist for all estimation problems with distinctive likelihoods so the maximum likelihood estimator class is only partially likelihood consistent. We therefore introduce the class of approximate maximum likelihood estimators \cite[See][Chapter 6]{balakrishnan2014order}, which has members for every estimation problem, and show that the class is strongly likelihood consistent. We in fact show that the approximate maximum likelihood estimator class characterises the class of estimation problems with distinctive likelihoods, in the sense that a problem has distinctive likelihoods if and only if all the approximate maximum likelihood estimators on the problem are strongly consistent.

\subsection{Maximum Likelihood}
\begin{defi}\label{D:distlike}
A large sample estimation problem $T=(\boldsymbol{\theta},\mathbf{x})$, has \emph{distinctive likelihoods} if, for all $\theta\in\Theta$ and neighbourhoods of $\theta$, $U$:
\begin{equation}\label{Eq:DLdef}
\prob\left(\limsup_{n\to\infty}\frac{\sup_{\theta'\notin U}{P_{\theta'}^{(1:n)}}(\mathbf{x}_{1:n})}{\sup_{\tilde{\theta}\in U}P_{\tilde{\theta}}^{(1:n)}(\mathbf{x}_{1:n})}<1\middle|\boldsymbol{\theta}=\theta\right)=1.
\end{equation}
\end{defi}

\begin{lemma}\label{L:Bin_DL}
Every binomial probability estimation problem with $p(\theta)=\theta$
has distinctive likelihoods.
\end{lemma}

\begin{proof}
Let $\rho(n)$ be as in Lemma~\ref{L:Bin_CLR}, and consider that
\[
P^{(1:n)}_{\theta'}(\mathbf{x}_{1:n})=\left(\theta'^{\rho(n)}(1-\theta')^{1-\rho(n)}\right)^n.
\]

Taken as a function of $\theta'$, this has a unique turning point: a maximum at
$\theta'=\rho(n)$. Recall, however, that by the strong law of large numbers
with probability $1$ the value of $\rho(n)$ converges to $\theta$, when
$\boldsymbol{\theta}=\theta$. Therefore, with probability $1$, it will
eventually be in any neighbourhood $U$ of $\theta$.
If, without loss of generality, we take $U$ to
be the interval $(\theta_1,\theta_2)$, then when $\rho(n)$ is in $U$,
\begin{align*}
\limsup_{n\to\infty} \frac{\sup_{\theta'\notin U} P^{(1:n)}_{\theta'}(\mathbf{x}_{1:n})}{\sup_{\tilde{\theta}\in U} P^{(1:n)}_{\tilde{\theta}}(\mathbf{x}_{1:n})}
&\le\limsup_{n\to\infty}\max\left(\frac{P^{(1:n)}_{\theta_1}(\mathbf{x}_{1:n})}{P^{(1:n)}_{\rho(n)}(\mathbf{x}_{1:n})},\frac{P^{(1:n)}_{\theta_2}(\mathbf{x}_{1:n})}{P^{(1:n)}_{\rho(n)}(\mathbf{x}_{1:n})}\right) \\
&\le\limsup_{n\to\infty}\max\left(\frac{P^{(1:n)}_{\theta_1}(\mathbf{x}_{1:n})}{P^{(1:n)}_{\theta}(\mathbf{x}_{1:n})},\frac{P^{(1:n)}_{\theta_2}(\mathbf{x}_{1:n})}{P^{(1:n)}_{\theta}(\mathbf{x}_{1:n})}\right),
\end{align*}
which we know from Lemma~\ref{L:Bin_CLR} to converge to $0$, and therefore to
eventually be less than $1$.
\end{proof}

\begin{thm}\label{T:MLlikecons}
Maximum Likelihood is strongly partially consistent over the class of problems with distinctive likelihoods.
\end{thm}
 \begin{proof}
 Suppose $T=(\boldsymbol{\theta},\mathbf{x})$ is a large sample estimation problem with distinctive likelihoods. Then, by definition, for any $\theta\in\Theta$, for $P_\theta$-almost every $x\in X$ there is some $N\in\setN$ dependent on $x$ such that for all $n>N$, $\frac{\sup_{\theta'\notin U}P_{\theta'}^{(1:n)}(x_{1:n})}{\sup_{\tilde{\theta}\in U}P_{\tilde{\theta}}^{(1:n)}(x_{1:n})}<1$. So for all $n>N$ and all $\theta'\notin U$, 
 \begin{equation}\label{eq:suplike}
 sup_{\tilde{\theta}\in\Theta}P_{\tilde{\theta}}^{(1:n)}(x_{1:n})>P^{(1:n)}_{\theta'}(x_{1:n})
 \end{equation}

Now suppose $\hat{\theta}_{\text{ML}}$ is a maximum likelihood estimator on $T$. Then \eqref{eq:suplike} implies that for any $\theta\in\Theta$ and neighbourhood, $U$, of $\theta$, there exists $N\in\setN$ such that for all $n>N$, $\argmax_{\tilde{\theta}\in\Theta}P_{\tilde{\theta}}^{(1:n)}(x_{1:n})\notin U^c$. Thus either $\hat{\theta}_{\text{ML}}$ either does not exist, or is in every neighbourhood of $\theta$ eventually. In other words, any maximum likelihood estimator on $T$ is strongly consistent.
 \end{proof}
 
The condition of being `partial' in the above result is necessary, because the likelihood function may fail to attain its supremum as we now show:

\begin{lemma}\label{L:MLnexist}
There exists a large sample estimation problem $T$ with distinctive likelihoods such that for all $n\in\setN$, and $x\in X$, $\argmax_{\theta\in\Theta}P_\theta^{(1:n)}(x_{1:n})=\emptyset$. Thus in particular there exists an estimation problem with distinctive likelihoods for which no maximum likelihood estimator exists.
\end{lemma}

\begin{proof}
 Let $T=(\boldsymbol{\theta},\mathbf{x})$ be a large sample estimation problem where $\Theta=\{\theta\in (0,1):\theta-\pi\in\mathbb{Q}\}$ and $\mathbf{x}_{1:n}$ is a Binomial random variable with probability parameter equal to $\boldsymbol{\theta}$.  Let $(k_n)_{n\in\setN}$ be a sequence of functions, with $k_n$ taking a sequence of $n$ observations and returning the number of ones observed. Thus $\mathbf{k}_n\defeq k_n(\mathbf{x}_{1:n})$ is a random variable representing the number of ones in the first $n$ observations (a sufficient statistic for this problem). The likelihood function given $\mathbf{k}_n=k$, is the restriction of the function $L(\eta;k,n)=\binom{n}{k}\eta^k (1-\eta)^{n-k}$ to $\Theta$. $L$ is differentiable in $\theta$ and has a single turning point at $\theta=\frac{k}{n}$ (a rational value) which is the unique global maximum. Hence, since $\Theta$ is dense in $[0,1]$ but contains no rational numbers, the likelihood function does not obtain its supremum over $\Theta$, and so, in particular, no maximum likelihood estimator exists for this problem.

From Lemma~\ref{L:Bin_DL}, however, $T$ is known to have distinctive
likelihoods.
\end{proof}
 
 \subsection{Approximate Maximum Likelihood}
Lemma \ref{L:MLnexist} provides our motivation for considering the class of \emph{approximate} Maximum Likelihood estimators:
 
 \begin{defi}\label{D:approxML}
 $(\hat{\theta}_n)_{n\in\setN}$ is an \emph{approximate maximum likelihood estimator} on estimation problem $(\boldsymbol{\theta},\mathbf{x})$ if:
\begin{equation*}
 \prob\left(\liminf_{n\to\infty} \frac{P_{\hat{\theta}_n}^{(1:n)}(\mathbf{x}_{1:n})}{\sup_{\theta'\in\Theta}P_{\theta'}^{(1:n)}(\mathbf{x}_{1:n})}=1\right)=1
 \end{equation*}
\end{defi}

 \begin{lemma}\label{L:existaML}
 For any large sample estimation problem $T$ there exists an approximate maximum likelihood estimator on $T$.
 \end{lemma}
 
 \begin{proof}
For any large sample estimation problem $(\boldsymbol{\theta},\mathbf{x})$, $x\in X$, $n\in\setN$, and $\epsilon>0$ there exists a $\theta^*(x_{1:n},\epsilon)\in\Theta$ such that $\eucNorm{\sup_{\theta'\in\Theta}P_\theta^{(1:n)}(x_{1:n})-P_{\theta^*(x_{1:n},\epsilon)}^{(1:n)}(x_{1:n})}<\epsilon$ by the definition of the supremum. Let $\hat{\theta}(x_{1:n})=\theta^*(x_{1:n},\frac{1}{n}\sup_{\theta'\in\Theta}P_{\theta'}^{(1:n)}(x_{1:n}))$. Then for all $x\in X$:
\begin{eqnarray*}
1\ge\liminf_{n\to\infty}\frac{P_{\hat{\theta}_n(x_{1:n})}^{(1:n)}(x_{1:n})}{\sup_{\theta'\in\Theta}P_{\theta'}^{(1:n)}(x_{1:n})}&\geq& \liminf_{n\to\infty}\frac{n-1}{n}\\
&=& 1.
\end{eqnarray*}

A fortiori 
\begin{equation*}
\prob\left(\liminf_{n\to\infty} \frac{P_{\hat{\theta}_n}^{(1:n)}(\mathbf{x}_{1:n})}{\sup_{\theta'\in\Theta}P_\theta^{(1:n)}(\mathbf{x}_{1:n})}=1\right)=1
\end{equation*}
so $(\hat{\theta}_n)_{n\in\setN}$ is an approximate maximum likelihood estimator.
\end{proof}
 
 \begin{lemma}\label{L:consaML}
 The class of approximate maximum likelihood estimators is strongly likelihood consistent.
 \end{lemma}
\begin{proof}
  Suppose $T=(\boldsymbol{\theta},\mathbf{x})$ is a large sample estimation problem with distinctive likelihoods and that $(\hat{\theta}_n)_{n\in\setN}$ is an approximate maximum likelihood estimator on $T$. Then, for each $\theta\in\Theta$, let $\mathscr{B}(\theta)$ be a countable neighbourhood basis of $\theta$ (which must exist since $\Theta$ is first countable).
  
  For a given $\theta^*\in\Theta$ and $U\in\mathscr{B}(\theta^*)$, because $T$ has distinctive likelihoods, 
  \begin{equation*}
  	\prob\left(\limsup_{n\to\infty}\frac{\sup_{\theta'\notin U}{P_{\theta'}^{(1:n)}}(\mathbf{x}_{1:n})}{\sup_{\tilde{\theta}\in U}P_{\tilde{\theta}}^{(1:n)}(\mathbf{x}_{1:n})}<1\middle|\boldsymbol{\theta}=\theta^*\right)=1,
  \end{equation*} 
and, by definition, if $\hat{\theta}$ is an approximate maximum likelihood estimator on $T$, then 
\begin{equation*}
\prob\left(\liminf_{n\to\infty} \frac{P_{\hat{\theta}_n(\mathbf{x}_{1:n})}^{(1:n)}(\mathbf{x}_{1:n})}{\sup_{\theta'\in\Theta}P_{\theta'}^{(1:n)}(\mathbf{x}_{1:n})}=1\right)=1.
\end{equation*}

Combining the above two equations we derive that with probability 1, when $\mathbf{x}$ is distributed according $P_\theta$:

\begin{equation}\label{eq:neweq}
\begin{aligned}
	\liminf_{n\to\infty} \frac{P_{\hat{\theta}_n(\mathbf{x}_{1:n})}^{(1:n)}(\mathbf{x}_{1:n})}{\sup_{\theta'\in\Theta}P_{\theta'}^{(1:n)}(\mathbf{x}_{1:n})}
&=\liminf_{n\to\infty} \frac{P_{\hat{\theta}_n(\mathbf{x}_{1:n})}^{(1:n)}(\mathbf{x}_{1:n})}{\sup_{\theta'\in U}P_{\theta'}^{(1:n)}(\mathbf{x}_{1:n})}\\
	&> \limsup_{n\to\infty}\frac{\sup_{\theta'\notin U}{P_{\theta'}^{(1:n)}}(\mathbf{x}_{1:n})}{\sup_{\tilde{\theta}\in U}P_{\tilde{\theta}}^{(1:n)}(\mathbf{x}_{1:n})}.
\end{aligned}
\end{equation}
Inequality \eqref{eq:neweq} implies that $\hat{\theta}_n$ is eventually in $U$ with probability 1 when $\mathbf{x}$ is distributed according to $P_\theta$. Therefore for all $\theta\in\Theta$ and $U$ in $\mathscr{B}(\theta)$ 
\begin{equation*}
\prob\left(\liminf_{n\to\infty} \mathbb{I}_{\{\hat{\theta}_n(\mathbf{x}_{1:n})\in U\}}\middle|\boldsymbol{\theta}=\theta\right)=1.
\end{equation*} 
Because $\mathscr{B}(\theta)$ is a countable neighbourhood basis, for all $\theta\in\Theta$, we conclude that
\begin{equation*}
\prob\left(\hat{\theta}_n(\mathbf{x}_{1:n})\rightarrow \boldsymbol{\theta}\middle|\boldsymbol{\theta}=\theta\right)=\prob\left(\bigcap_{U\in\mathscr{B}(\theta)}\liminf_{n\to\infty} \mathbb{I}_{\{\hat{\theta}_n(\mathbf{x}_{1:n})\in U\}}\middle|\boldsymbol{\theta}=\theta\right)=1.
\end{equation*}
That is $\hat{\theta}$ is strongly consistent.
\end{proof}
	
  \begin{cor}\label{C:distimpconv}
 Every estimation problem with distinctive likelihoods has a consistent posterior.
 \end{cor}
 \begin{proof}
 By Lemmas \ref{L:existaML} and  \ref{L:consaML} every problem with distinctive likelihoods has a consistent estimator, and by Lemma \ref{L:impcons} every estimation problem with a consistent estimator has a consistent posterior.
 \end{proof}

 \begin{thm}\label{T:approxML}
 Large sample estimation problem $T$ has distinctive likelihoods if and only if all approximate maximum likelihood estimators on $T$ are strongly consistent.
 \end{thm}
 
 \begin{proof}
 The `only if' part of this result is Lemma \ref{L:consaML}. For the `if' we suppose $T=(\boldsymbol{\theta},\mathbf{x})$ does not have distinctive likelihoods and will show that there exists an approximate maximum likelihood estimator which is not strongly consistent on $T$.
 
 Since $T$ does not have distinctive likelihoods there is some $\theta^*\in \Theta$, and neighbourhood $U^*$ of $\theta^*$ such that with positive probability when $\mathbf{x}$ is distributed according to $P_{\theta^*}$:  
 $\limsup_{n\to\infty}\frac{\sup_{\theta'\notin U^*}P_{\theta'}^{(1:n)}(\mathbf{x}_{1:n})}{\sup_{\tilde{\theta}\in U^*}P_{\tilde{\theta}}^{(1:n)}(\mathbf{x}_{1:n})}\geq 1$. Let $B^*=\{x\in X: \limsup_{n\to\infty}\frac{\sup_{\theta'\notin U^*}P_{\theta'}^{(1:n)}(x_{1:n})}{\sup_{\tilde{\theta}\in U^*}P_{\tilde{\theta}}^{(1:n)}(x_{1:n})}\geq 1\}$, so that the previous sentence can be rewritten more compactly as:
\begin{equation}\label{eq:ndistlike}
\prob(\mathbf{x}\in B^*|\boldsymbol{\theta}=\theta^*)>0.
\end{equation}

 Now, by Lemma \ref{L:existaML} there exists an approximate maximum likelihood estimator on $T$, $\hat{\theta}$. Using $\hat{\theta}$, we construct an approximate maximum likelihood estimator, $\hat{\theta}^*$, on $T$ which is not strongly consistent, together with a function $k_n:X_{1:n}\rightarrow \setN$ recursively on $n$ as follows:

\begin{eqnarray*}
k(x_1)&=&1\\
\hat{\theta}^*(x_1)&=&\hat{\theta}(x_1)\\
 \end{eqnarray*}
 and, for $n\in\setN$, if there exists a $\theta'\in U^{*c}$ such that  $\frac{P_{\theta'}(x_{1:n+1})}{\sup_{\tilde{\theta}\in \Theta}P_{\tilde{\theta}}(x_{1:n+1})}>\frac{k(x_{1:n})-1}{k(x_{1:n})}$, set $\hat{\theta}^*(x_{1:n+1})$ to be such a $\theta'$ and $k(x_{1:n+1})=k(x_{1:n})+1$.
 
 Otherwise set: 
 \begin{eqnarray*}
k(x_{1:n+1})&=&k(x_{1:n})\\
 \hat{\theta}^*(x_{1:n+1})&=&\hat{\theta}(x_{1:n+1}),
 \end{eqnarray*}

By the axiom of choice an estimator exists satisfying this recursive definition. According to this construction $\hat{\theta}^*$ is an approximate maximum likelihood estimator since if $\lim_{n\to\infty}k(x_{1:n})<\infty$, then $\hat{\theta^*}=\hat{\theta}$ eventually, and if $\lim_{n\to\infty}k(x_{1:n})=\infty$, then for those $n$ such that $k(x_{1:n+1})=1+k(x_{1:n})$, by construction $\frac{P_{\hat{\theta}^*(x_{1:n+1})}(x_{1:n+1})}{\sup_{\tilde{\theta}\in\Theta} P_{\tilde{\theta}}(x_{1:n+1})}>\frac{k(x_{1:n})-1}{k(x_{1:n})}$ which converges to 1 as $n$ (and therefore $k(x_{1:n})$) goes to infinity. Secondly, if $x\in B^*$ then, for infinitely many $n$, $\sup_{\tilde{\theta}\in \Theta}P_{\tilde{\theta}}^{(1:n)}(x_{1:n})=\sup_{\tilde{\theta}\notin U^*}P_{\tilde{\theta}}^{(1:n)}(x_{1:n})$. Therefore $\hat{\theta}^*(x_{1:n})$ is in $U^{*C}$ infinitely often, implying that $\hat{\theta}^*$ does not converge to $\theta^*$ so: $\prob(\hat{\theta}^*\nrightarrow \theta^*|\boldsymbol{\theta}=\theta^*)\geq\prob(\mathbf{x}\in B^*|\boldsymbol{\theta}=\theta^*)>0$ (by \eqref{eq:ndistlike}). Therefore $\hat{\theta}^*$ is an approximate maximum likelihood estimator which is not strongly consistent on $T$.
 
 \end{proof}
 
\begin{thm} There exist large sample estimation problems on which all approximate maximum likelihood estimators are consistent but which do not have distinctive likelihoods.
\end{thm}

\begin{proof}
Consider a large sample estimation problem $T=(\boldsymbol{\theta},\mathbf{x})$ such that $\Theta=\setN\cup\{0\}$, and $X_i=\{0,1,2\}$, where $(\mathbf{x}_n)_{n\in\setN}$ is a sequence of independent random variables when conditioned on $\boldsymbol{\theta}$, with conditional distributions given by:

\begin{equation*}
\prob(\mathbf{x}_n=x|\boldsymbol{\theta}=0)=\left\{\begin{array}{ll}
\frac{n-1}{n} & \text{if }x=0\\
\frac{1}{n} & \text{if }x=1\\
0 & \text{if }x=2,
\end{array}
\right.
\end{equation*}

and for $k\in\setN$:

\begin{equation*}
\prob(\mathbf{x}_n=x|\boldsymbol{\theta}=k)=\left\{\begin{array}{ll}
1 & \text{if }x=1\text{ and either }n=k\text{ or }n=1\\
\left(\frac{n-1}{n}\right)^2 & \text{if }1<n<k\text{ and }x=0\\
\frac{n-1}{n^2} & \text{if }1<n<k\text{ and }x=1\\
\frac{1}{n} & \text{if }1<n<k\text{ and }x=2\\
1 & \text{if }n>k\text{ and }x=2\\
0 & \text{otherwise}.
\end{array}
\right.
\end{equation*}

The critical finite observation sequences to consider in this estimation problem are those consisting of only zeroes and ones, which we will divide into those where the final observation, $x_n$, is 0 and those where it is 1. Note that the conditional distributions have been defined above so that when $x_i\in\{0,1\}$ and $1<i<k$, $\prob(\mathbf{x}_i=x_i|\boldsymbol{\theta}=k)=\frac{i-1}{i}\prob(\mathbf{x}_i=x_i|\boldsymbol{\theta}=0)$. Therefore, when $x_{1:n}$ is a finite observation sequence of zeroes and ones, and $n<k$ the likelihood $P_\theta^{(1:n)}(x_{1:n})$ is given by
\begin{equation}\label{eq:end0}
\begin{aligned}
P_k^{(1:n)}(x_{1:n})&=\prod_{i=2}^n \left(\frac{i-1}{i}\right)P_0^{(1:n)}(x_{1:n})\\
&=\frac{1}{n}(P_0^{(1:n)}(x_{1:n})).
\end{aligned} 
\end{equation}

Note that for $n=1$, \eqref{eq:end0} still holds, since $P_\theta^{(1)}(1)=1$ for all $\theta\in\Theta$.

For a finite observation sequences of length $n$ ending in 0 the only possible values of $\theta$ with positive likelihood are 0 and those greater than $n$. If a finite observation sequence ends in $1$ then  $\theta=n$ also has positive likelihood, which is $1$ if $n=1$ and otherwise is given by the following equation.
\begin{equation}\label{eq:end1}
\begin{aligned}
P_n^{(1:n)}(x_{1:n})&=P_n^{(1:n-1)}(x_{1:n-1})\\
&=\frac{1}{n-1}(P_0^{(1:n-1)}(x_{1:n-1}))\\
&=\frac{n}{n-1}(P_0^{(1:n)}(x_{1:n})).
\end{aligned}
\end{equation}

Therefore the maximum likelihood estimate of $\theta$ is 0 in the case of finite observation sequences of zeroes and ones ending in 0, and is $n$ in the case of finite observation sequences of zeroes and ones, of length at least 2, ending in 1 (for $n=1$ every member of $\Theta$ has maximum likelihood). Now, $\sum_{i=1}^\infty P^i_0(1)=\sum_{i=1}^\infty \frac{1}{i}=\infty$, so, by the second Borel-Cantelli lemma, $\prob(\limsup_{n\to\infty} \{\mathbf{x}_n=1\}|\boldsymbol{\theta}=0)=1$. Therefore, when $\theta=0$ the maximum likelihood estimate is non-zero infinitely often, so the maximum likelihood estimate, while it exists, is not strongly consistent. Thus, by Theorem \ref{T:MLlikecons}, $T$ does not have distinctive likelihoods.

It remains to show that all approximate maximum likelihood estimates are consistent on $T$. We show that all approximate maximum likelihood estimators converge in probability to $\boldsymbol{\theta}$ by dividing into the cases where $\boldsymbol{\theta}=0$ and where $\boldsymbol{\theta}\ne 0$. Considering the latter case first, note that when $\boldsymbol{\theta}=k\ne 0$, $\mathbf{x}_{k:\infty}=(1,2,2,2,\ldots)$ with probability 1, so: \begin{equation}
\lim_{n\to\infty}P_k^{(1:n)}(\mathbf{x}_{1:n})=P_k^{(1:k-1)}(\mathbf{x}_{1:k-1})>0.\label{eq:Pk+}
\end{equation} 
Also for any $i$ less than $k$, $P_i^{(k)}(1)=0$ and for any $i>k$, $P_i^{(i)}(2)=0$, while $P_k^{(k)}(1)=P_k^{(i)}(2)=1$ for all $i>k$. Therefore when $\boldsymbol{\theta}=k$, $\theta\leq n$, $k\leq n$ and $\theta\ne k$, $P_\theta^{(1:n)}(\mathbf{x}_{1:n})=0$ with probability 1. Hence, with probability 1 when $\mathbf{x}$ is distributed according to $P_k$:
\begin{equation}\label{eq:Psup0}
\begin{aligned}
\lim_{n\to\infty}\sup_{\theta\in\Theta\setminus\{k\}}P_\theta^{(1:n)}(\mathbf{x}_{1:n})
&=\lim_{n\to\infty}\sup_{\theta>n}P_\theta^{(1:n)}(\mathbf{x}_{1:n})\\
&\leq\lim_{n\to\infty}(P_k^{(1:k-1)}(\mathbf{x}_{1:k-1}))(P_{n+1}^{(k)}(1))(P_{n+1}^{(k+1)}(2))^{n-k}\\
&=0.
\end{aligned}
\end{equation}

Combining \eqref{eq:Pk+} and \eqref{eq:Psup0}, for any $k\ne 0$: 
\begin{equation*}
\prob\left(\liminf_{n\to\infty}\frac{\sup_{\theta\ne k} P_\theta^{(1:n)}(\mathbf{x}_{1:n})}{P_k^{(1:n)}(\mathbf{x}_{1:n})}=0\middle|\boldsymbol{\theta}=k\right)=1,
\end{equation*} 
and so in particular also
\begin{equation*}
\prob\left(\liminf_{n\to\infty}\frac{\sup_{\theta\ne k} P_\theta^{(1:n)}(\mathbf{x}_{1:n})}{\sup_{\theta\in\Theta} P_\theta^{(1:n)}(\mathbf{x}_{1:n})}=0\middle|\boldsymbol{\theta}=k\right)=1,
\end{equation*} 
because it is always the case that $P_k^{(1:n)}(\mathbf{x}_{1:n})\le \sup_{\theta\in\Theta} P_\theta^{(1:n)}(\mathbf{x}_{1:n})$.

Now, if $\theta_n$ is a sequence in $\Theta$ not eventually equal to $k$, then, with probability $1$ when $\boldsymbol{\theta}=k$,

\begin{equation*}
\liminf_{n\to\infty}\frac{P_{\theta_n}^{(1:n)}(\mathbf{x}_{1:n})}{ \sup_{\theta\in\Theta} P_\theta^{(1:n)}(\mathbf{x}_{1:n})}\leq\liminf_{n\to\infty}\frac{\sup_{\theta\ne k}P_{\theta}^{(1:n)}(\mathbf{x}_{1:n})}{ \sup_{\theta\in\Theta} P_\theta^{(1:n)}(\mathbf{x}_{1:n})}=0.
\end{equation*} 

By contrast, for any approximate maximum likelihood estimator $\hat{\theta}_n$ it is the case by Definition~\ref{D:approxML} that when $\boldsymbol{\theta}=k$
then, with probability $1$,
\begin{equation*} \liminf_{n\to\infty}\frac{P_{\hat{\theta}_n(\mathbf{x}_{1:n})}^{(1:n)}(\mathbf{x}_{1:n})}{\sup_{\theta\in\Theta}P_\theta^{(1:n)}(\mathbf{x}_{1:n})}=1.
\end{equation*}
Thus, for any $k>0$, every approximate maximum likelihood estimator must eventually be equal to $k$ with probability 1 conditioned on $\boldsymbol{\theta}=k$. A fortiori, every approximate maximum likelihood estimator converges in probability to $\boldsymbol{\theta}$ when $\boldsymbol{\theta}\ne 0$.  

When $\boldsymbol{\theta}=0$, $\mathbf{x}_n\in\{0,1\}$ for all $n\in\setN$. Note that by the second Borel-Cantelli lemma, $\prob(\limsup_{n\to\infty} \{\mathbf{x}_n=0\}|\boldsymbol{\theta}=0)=1$. Thus, let $(\mathbf{x}_{\mathbf{n}_m})_{m\in\setN}$ be the infinite subsequence of $(\mathbf{x}_n)_{n\in\setN}$ such that $\mathbf{x}_{\mathbf{n}_m}=0$ for all $m\in\setN$. Thus for all $m\in\setN$, $\mathbf{x}_{1:\mathbf{n}_m}$ is a finite observation sequence of zeroes and ones ending in 0.  Above we noted that given such a sequence of length $n$, the only possible values for $\theta$ with positive likelihood are 0 and those values greater than $n$. Further the ratio between the likelihood of 0 and that of any value greater than $n$ is given by \eqref{eq:end0}. Therefore with probability 1 when conditioned on $\boldsymbol{\theta}=0$  
\begin{equation*}
\liminf_{k\to\infty}\frac{\sup_{\theta\ne 0}P_\theta^{(1:n_k)}(\mathbf{x}_{1:n_k})}{P_0^{(1:n_k)}(\mathbf{x}_{1:n_k})}\leq\liminf_{k\to\infty}\frac{1}{\mathbf{n}_k}=0.
\end{equation*}
Therefore, for any approximate maximum likelihood estimator $\hat{\theta}$, $\hat{\theta}_{n_k}$ is eventually equal to 0, when conditioned on $\boldsymbol{\theta}=0$. Therefore: 
\begin{eqnarray*}
	\lim_{n\to\infty} \prob(\hat{\theta}=0|\boldsymbol{\theta}=0)&\geq& \lim_{n\to\infty}\prob(\mathbf{x}_n=0|\boldsymbol{\theta}=0) \\
	&=&\lim_{n\to\infty}\frac{n-1}{n}=1.
\end{eqnarray*}
Thus, for any approximate maximum likelihood estimator $\hat{\theta}$:
\begin{eqnarray*}
	\lim_{n\to\infty}\prob(\hat{\theta}_n=\boldsymbol{\theta})&=&\prob(\boldsymbol{\theta}=0)\lim_{n\to\infty}\prob(\hat{\theta}_n=0|\boldsymbol{\theta}=0)\\
	&&\quad+\prob(\boldsymbol{\theta}\ne 0)\lim_{n\to\infty}\prob(\hat{\theta}_n=\boldsymbol{\theta}|\boldsymbol{\theta}\ne 0)\\
		&=& 1
\end{eqnarray*}
That is, all approximate maximum likelihood estimators are consistent on $T$, despite $T$ not having distinctive likelihoods. 
\end{proof}

\section{Strict Minimum Message Length}\label{S:SMML}
In the discrete case SMML is ordinarily defined as the (or an) estimator with the minimum \emph{expected message length}, as defined below \cite[p. 155]{Wallace2005}:

\begin{defi}\label{D:messageLength}
The \emph{expected message length} of an estimator $\hat{\theta}_n$ given $n$ data points is:\footnote{For readers familiar with Wallace's notation of \cite{Wallace2005}, we give a translation: $I_1(\hat{\theta}_n)=I_1$ (the only difference being Wallace's use of $I_1$ leaves reference to a particular estimator implicit), $x_{1:n}=x$ (again the only difference here is that we make the number of observations explicit), $\prob(\mathbf{x}_{1:n}=x_{1:n})=r(x)$, $\hat{\theta}_n(x_{1:n})=m(x)$, and finally $P^{(1:n)}_{\theta}(x_{1:n})=f(x|\theta)$, so that $P_{\hat{\theta}_n(x_{1:n})}^{(1:n)}(x_{1:n})=f(x|m(x))$.}
\begin{equation*}\label{eq:messageLength}
I_1(\hat{\theta}_n(\mathbf{x}_{1:n}))\defeq H(\hat{\theta}_n)\,\ -\!\!\sum_{x_{1:n}\in X_{1:n}}\prob(\mathbf{x}_{1:n}=x_{1:n})\log(P^{(1:n)}_{\hat{\theta}_n(x_{1:n})}(x_{1:n})),
\end{equation*}
where $H$ represents Shannon's entropy.\footnote{As is commonly done in information theory, here and everywhere we take any expression evaluating to ``$0 \log 0$'' to equal $0$.}
\end{defi}

By Shannon's source coding theorem \cite{shannon1948}, $I_1(\hat{\theta}_n)$ is a lower bound on the expected length of a two-part \emph{explanation message}, conveying first the estimate $\hat{\theta}_n$ using an encoding designed only given knowledge of the estimation problem and estimation method, and then conveying the observations $\mathbf{x}_{1:n}$ using an encoding which is efficient assuming $\mathbf{x}_{1:n}$ is distributed according to $P_{\hat{\theta}_n}^{(1:n)}(\mathbf{x}_{1:n})$. By the same theorem this is necessarily greater than or equal to the expected length of a one-part message conveying just the observations $\mathbf{x}_{1:n}$. 

One consequence of this is that if $H(\mathbf{x}_{1:n})$ is infinite, then $I_1$ is also infinite for every estimator, and so cannot be used to compare estimators. We thus consider the standard definition of SMML in the discrete case in terms of $I_1$ to be implicitly restricted to cases where the observations have finite entropy. Since we wish to make no such restriction we shall instead define SMML in terms of the \emph{excess expected message length} analogously to the standard treatment of the continuous case as in \cite[][p.~168]{Wallace2005} or \cite[][p.~3]{brand2016neymanscott}:

\begin{defi}\label{D:eMessageLength}
The \emph{excess expected message length} of an estimator $\hat{\theta}_n$ is given by the equation:
\begin{equation*}
I_2(\hat{\theta}_n)\defeq H(\hat{\theta}_n(\mathbf{x}_{1:n}))\,\  + \!\!\sum_{x_{1:n}\in X_{1:n}}\prob(\mathbf{x}_{1:n}=x_{1:n})\log\left(\frac{\prob(\mathbf{x}_{1:n}=x_{1:n})}{P^{(1:n)}_{\hat{\theta}_n(x_{1:n})}(x_{1:n})}\right)
\end{equation*}
\end{defi}

We now define the SMML estimator class by minimisation of $I_2$.

\begin{defi}\label{D:SMML}
An estimator $(\hat{\theta}_n)$ is a \emph{strict minimum message length} estimator if, for all $n\in\setN$:
\begin{equation*}
I_2(\hat{\theta}_n)=\inf_{F\in \Theta^{X_{1:n}}}I_2(F),
\end{equation*}
where $\Theta^{X_{1:n}}$ is the set of all functions from $X_{1:n}$ to $\Theta$ - that is, the set of all estimators for $n$ observations.
\end{defi}

Note that when $H(\mathbf{x}_{1:n})$ is finite $I_2(\hat{\theta}_n)$ equals $I_1(\hat{\theta}_n)-H(\mathbf{x}_{1:n})$, and because $H(\mathbf{x}_{1:n})$ does not depend on $\hat{\theta}_n$ the two functions necessarily reach their minimum at the same $\hat{\theta}_n$ values. When $H(\mathbf{x}_{1:n})<\infty$, this definition therefore agrees with that of \cite{Wallace2005}, while when $H(\mathbf{x}_{1:n})=\infty$ this definition remains applicable even though the definition of \cite{Wallace2005} does not.

 Intuitively, $I_2(\hat{\theta}_n)$ gives  the expected difference between an explanation message as described above, and an efficient one part message, encoding the same data. Since it is always possible to reconstruct the data from an explanation message, $I_2$ is bounded below by 0. We show this below along with the conditions under which this lower bound for $I_2$ is obtained. We  use these conditions to construct estimation problems for which the SMML estimators have 0 excess expected message length, simplifying the task of analysing their consistency. The proof relies on the well-known \emph{Gibbs inequality}:
 
 \begin{lemma}[Gibbs' Inequality]\label{L:Gibbs}
If $\mathbf{P}$ and $\mathbf{Q}$ are probability distributions over $\Omega$ ($\Omega$ countable) then
\begin{equation*}
\sum_{\omega\in\Omega} \mathbf{P}(\omega)\log\left(\frac{\mathbf{P}(\omega)}{\mathbf{Q}(\omega)}\right)\geq 0,
\end{equation*}
with equality if and only if $\mathbf{P}(\omega)=\mathbf{Q}(\omega)$ for all $\omega\in\Omega$.
\end{lemma}

\begin{lemma}\label{L:SMML}
\begin{equation*}
I_2(\hat{\theta}_n)\geq 0
\end{equation*}
with equality if and only if, for every $\theta\in\Theta$ such that $\prob(\hat{\boldsymbol{\theta}}_n=\theta)>0$, for every $x_{1:n}$ such that $\hat{\theta}_n(x_{1:n})=\theta$,
\begin{equation}\label{eq:condeqlike}
\prob(\mathbf{x}_{1:n}=x_{1:n}|\boldsymbol{\hat{\theta}}_n=\theta)=P_{\theta}^{(1:n)}(x_{1:n}).
\end{equation}
\end{lemma}

\begin{proof}
		
Let $\Theta^*=\{\theta\in\Theta:\prob(\hat{\boldsymbol{\theta}}_n=\theta)>0\}$.
For $\theta\in\Theta^*$,
let $\hat{\theta}_n^{-1}(\theta)=\{x_{1:n}\in X_{1:n}|\hat{\theta}_n(x_{1:n})=\theta\}$ and note that for all $x_{1:n}\in\hat{\theta}^{-1}(\theta)$, $\prob(\mathbf{x}_{1:n}=x_{1:n}|\hat{\boldsymbol{\theta}}_n=\theta)=\prob(\mathbf{x}_{1:n}=x_{1:n})/\prob(\hat{\boldsymbol{\theta}}_n=\theta)$, and furthermore that
\[
\prob(\mathbf{x}_{1:n}=x_{1:n})=\sum_{\theta\in\Theta^*}\prob(\hat{\boldsymbol{\theta}}_n=\theta)\prob(\mathbf{x}_{1:n}=x_{1:n}|\hat{\boldsymbol{\theta}}_n=\theta).
\]

Using this, we get
\begin{equation}\label{eq:SMMLGibbs}
\begin{aligned}
I_2(\hat{\theta}_n)&=H(\hat{\theta}_n(\mathbf{x}_{1:n})) + \!\!\!\sum_{x_{1:n}\in X_{1:n}}\!\!\!\prob(\mathbf{x}_{1:n}=x_{1:n})\log\left(\frac{\prob(\mathbf{x}_{1:n}=x_{1:n})}{P^{(1:n)}_{\hat{\theta}_n(x_{1:n})}(x_{1:n})}\right)\\
&=-\sum_{\theta\in\Theta^*} \prob(\hat{\boldsymbol{\theta}}_n=\theta)\log(\prob(\hat{\boldsymbol{\theta}}_n=\theta))+\sum_{\theta\in\Theta^*}\prob(\hat{\boldsymbol{\theta}}_n=\theta)\left[\vphantom{\sum_{x_{1:n}\in\hat{\theta}^{-1}_n(\theta)}\prob(\mathbf{x}_{1:n}=x_{1:n}|\hat{\boldsymbol{\theta}}_n=\theta)\log\left(\frac{\prob(\mathbf{x}_{1:n}=x_{1:n})}{P^{(1:n)}_{\hat{\theta}_n(x_{1:n})}(x_{1:n})}\right)}\right. \\
&\qquad\qquad\qquad\qquad \left.\sum_{x_{1:n}\in\hat{\theta}^{-1}_n(\theta)}\prob(\mathbf{x}_{1:n}=x_{1:n}|\hat{\boldsymbol{\theta}}_n=\theta)\log\left(\frac{\prob(\mathbf{x}_{1:n}=x_{1:n})}{P^{(1:n)}_{\hat{\theta}_n(x_{1:n})}(x_{1:n})}\right)\right]\\
&=\sum_{\theta\in\Theta^*} \prob(\hat{\boldsymbol{\theta}}_n=\theta)\left[\sum_{x_{1:n}\in\hat{\theta}^{-1}_n(\theta)}\prob(\mathbf{x}_{1:n}=x_{1:n}|\hat{\boldsymbol{\theta}}_n=\theta)\vphantom{\log\left(\frac{\prob(\mathbf{x}_{1:n}=x_{1:n}|\hat{\boldsymbol{\theta}}_n=\theta)}{P^{(1:n)}_{\hat{\theta}_n(x_{1:n})}(x_{1:n})}\right)}\right.\\
&\qquad\qquad\qquad\qquad\qquad\qquad\qquad\qquad\left.\log\left(\frac{\prob(\mathbf{x}_{1:n}=x_{1:n}|\hat{\boldsymbol{\theta}}_n=\theta)}{P^{(1:n)}_{\hat{\theta}_n(x_{1:n})}(x_{1:n})}\right)\right]
\end{aligned}
\end{equation}

By Gibbs' inequality the inner sum in \eqref{eq:SMMLGibbs} is non-negative, and is 0 for exactly those values of $\theta$ such that equation \eqref{eq:condeqlike} holds for every $x_{1:n}$ such that $\hat{\theta}_n(x_{1:n})=\theta$. Therefore $I_2(\hat{\theta}_n)=0$ if and only if \eqref{eq:condeqlike} holds for all $\theta\in\Theta^*$.
\end{proof}

\begin{cor}\label{C:SMML} If $I_2(\hat{\theta}_n)=0$ and $\theta_1,\theta_2\in\Theta$ are distinct, and both $\prob(\hat{\boldsymbol{\theta}}_n=\theta_1)$ and $\prob(\hat{\boldsymbol{\theta}}_n=\theta_2)$ are positive, then $\supp\left(P_{\theta_1}^{(1:n)}\right)\cap\supp\left(P_{\theta_2}^{(1:n)}\right)=\emptyset$.

Therefore $\{\supp\left(P_\theta^{(1:n)}\right):\prob(\hat{\boldsymbol{\theta}}_n=\theta)>0\}$ partitions $\supp(\mathbf{x}_{1:n})$. 
\end{cor}

\begin{proof}
For any $\theta\in\Theta$ if $\prob(\hat{\boldsymbol{\theta}}_n\!=\!\theta)\!>\!0$ and $I_2(\hat{\theta}_n)=0$ then, by Lemma \ref{L:SMML},
\begin{equation*}
\prob(\mathbf{x}_{1:n}=x_{1:n}|\hat{\boldsymbol{\theta}}_n=\theta)=P_{\theta}^{(1:n)}(x_{1:n}).
\end{equation*}
More specifically, for such $\theta$, $\prob(\mathbf{x}_{1:n}\!=\!x_{1:n}|\hat{\boldsymbol{\theta}}_n\!=\!\theta)\!>\!0$ if and only if $P_{\theta_1}^{(1:n)}(x_{1:n})\!>\!0$. But 	$\prob(\mathbf{x}_{1:n}\!=\!x_{1:n}|\hat{\boldsymbol{\theta}}_n\!=\!\theta)\!>\!0$ if and only if $\hat{\theta}_n(x_{1:n})\!=\!\theta$. Therefore, $\hat{\theta}_n(x_{1:n})\!=\theta$ if and only if $P_\theta^{(1:n)}(x_{1:n})>0$.

 Therefore, if $\theta_1,\theta_2\in\Theta$ are distinct, both $\prob(\hat{\boldsymbol{\theta}}=\theta_1)$ and $\prob(\hat{\boldsymbol{\theta}}=\theta_2)$ are positive and $x_{1:n}\in\supp\left(P_{\theta_1}^{(1:n)}\right)\cap\supp\left(P_{\theta_2}^{(1:n)}\right)$, then $\hat{\theta}_n(x_{1:n})=\theta_1$ and $\hat{\theta}_n(x_{1:n})=\theta_2$ contradicting the distinctness of $\theta_1$ and $\theta_2$.

 We conclude that $\supp\left(P_{\theta_1}^{(1:n)}\right)\cap\supp\left(P_{\theta_2}^{(1:n)}\right)=\emptyset$. It follows immediately that $\{\supp\left(P_\theta^{(1:n)}\right):\prob(\hat{\boldsymbol{\theta}}_n=\theta)>0\}$ partitions $\supp(\mathbf{x}_{1:n})$ because for every $x_{1:n}\in\supp(\mathbf{x}_{1:n})$, $\prob(\hat{\boldsymbol{\theta}}_n=\hat{\theta}_n(x_{1:n}))>0$ and $x_{1:n}\in \supp\left(P_{\hat{\theta}_n(x_{1:n})}^{(1:n)}\right)$.
\end{proof} 

We introduce SMML as an example for an estimator class that does not guarantee
consistency even in the case of problems with distinctive likelihoods.
Specifically, we prove the following.

\begin{thm}\label{T:SMMLinc}
The class of SMML estimators is not partially consistent over the class of estimation problems with distinctive likelihoods. In fact there exist estimation problems with distinctive likelihoods on which all SMML estimators are inconsistent.
\end{thm}

\begin{proof}
Let $T=(\boldsymbol{\theta},\mathbf{x})$ be a large sample estimation problem, with $\Theta=\setN$, and $X_i=\{0,1,2\}$ for all $i\in\setN$. Let the prior be defined by the probabilities $\prob(\boldsymbol{\theta}=\theta)=2^{-\theta}$, and the probabilities for the observations be given by:
\begin{equation*}
\prob\left(\mathbf{x}_n=x_n|\boldsymbol{\theta}=\theta,\mathbf{x}_{1:n-1}=x_{1:n-1}\right)=\left\{\begin{array}{ll}
\frac{1}{\theta} & \text{if }n> 2\lceil\frac{\theta}{2}\rceil, x_n=1\\[2pt]
\frac{\theta-1}{\theta} & \text{if }n>2\lceil\frac{\theta}{2}\rceil, x_n=0\\[2pt]
1 & \text{if }n\in\{\theta,2\lceil\frac{\theta}{2}\rceil\}, x_n=2\\[2pt]
p(n, x_n, x_{1:n-1}) & \text{if }n<\theta\\[2pt]
0 & \text{otherwise}
\end{array}
\right.
\end{equation*}
Where
\begin{equation}\label{eq:postp}
p(n, x_n,x_{1:(n-1)})=\prob(\mathbf{x}_n=x_n|\boldsymbol{\theta}\leq n,\mathbf{x}_{1:(n-1)}=x_{1:(n-1)}).
\end{equation}

The distribution of $\mathbf{x}_n$ conditioned on $\boldsymbol{\theta}=\theta$ for $n,\theta\leq 10$ is represented in Table \ref{Table:SMML}. 
\renewcommand{\arraystretch}{1.3}
\newcolumntype{C}[1]{>{\centering\arraybackslash}p{#1}}
\newlength{\mx}
\settowidth{\mx}{$p(3,\cdot, \mathbf{x}_{1:2})$}
\begin{table}[bh]
	\centering
	\resizebox{\textwidth}{!}{%
		\begin{tabular}{|c|c|C{\mx}C{\mx}C{\mx}C{\mx}C{\mx}C{\mx}C{\mx}C{\mx}C{\mx}C{\mx}}
			\cline{3-12}
			\multicolumn{1}{c}{}&& \multicolumn{10}{c}{n}\\
			\cline{3-12}
			\multicolumn{1}{c}{}&& 1 & 2 & 3 & 4 & 5 & 6 & 7 & 8 & 9 & 10\\
			\hline
			\multirow{10}{*}{$\theta$} &  1 & \textbf{2} & \textbf{2} & \textbf{1} & \textbf{1} & \textbf{1} & \textbf{1} & \textbf{1} & \textbf{1} & \textbf{1} & \textbf{1}\\
			& 2 & $p(1,\cdot)$ & \textbf{2} & $\frac{1}{2}$ & $\frac{1}{2}$ & $\frac{1}{2}$ & $\frac{1}{2}$ & $\frac{1}{2}$ & $\frac{1}{2}$ & $\frac{1}{2}$ & $\frac{1}{2}$\\
			\cdashline{2-12}
			& 3 & $p(1,\cdot)$ & $p(2,\cdot,\mathbf{x}_1)$ & \textbf{2} & \textbf{2} & $\frac{1}{3}$ & $\frac{1}{3}$ & $\frac{1}{3}$ & $\frac{1}{3}$ & $\frac{1}{3}$ & $\frac{1}{3}$\\
			& 4 & $p(1,\cdot)$ & $p(2,\cdot,\mathbf{x}_1)$ & $p(3,\cdot, \mathbf{x}_{1:2})$ & \textbf{2} & $\frac{1}{4}$ & $\frac{1}{4}$ & $\frac{1}{4}$ & $\frac{1}{4}$ & $\frac{1}{4}$ & $\frac{1}{4}$\\
			\cdashline{2-12}
			& 5 & $p(1,\cdot)$ & $p(2,\cdot,\mathbf{x}_1)$ & $p(3,\cdot, \mathbf{x}_{1:2})$ & $p(4,\cdot,\mathbf{x}_{1:3})$ & \textbf{2} & \textbf{2} & $\frac{1}{5}$ & $\frac{1}{5}$ & $\frac{1}{5}$ & $\frac{1}{5}$ \\
			& 6 & $p(1,\cdot)$ & $p(2,\cdot,\mathbf{x}_1)$ & $p(3,\cdot, \mathbf{x}_{1:2})$ & $p(4,\cdot,\mathbf{x}_{1:3})$  & $p(5,\cdot,\mathbf{x}_{1:4})$ & \textbf{2} & $\frac{1}{6}$ & $\frac{1}{6}$ & $\frac{1}{6}$ & $\frac{1}{6}$\\
			\cdashline{2-12}
			& 7 & $p(1,\cdot)$ & $p(2,\cdot,\mathbf{x}_1)$ & $p(3,\cdot, \mathbf{x}_{1:2})$ & $p(4,\cdot,\mathbf{x}_{1:3})$  & $p(5,\cdot,\mathbf{x}_{1:4})$ & $p(6,\cdot,\mathbf{x}_{1:5})$ & \textbf{2} & \textbf{2} & $\frac{1}{7}$ & $\frac{1}{7}$ \\
			& 8 & $p(1,\cdot)$ & $p(2,\cdot,\mathbf{x}_1)$ & $p(3,\cdot, \mathbf{x}_{1:2})$ & $p(4,\cdot,\mathbf{x}_{1:3})$  & $p(5,\cdot,\mathbf{x}_{1:4})$ & $p(6,\cdot,\mathbf{x}_{1:5})$ & $p(7,\cdot,\mathbf{x}_{1:6})$ & \textbf{2} & $\frac{1}{8}$ & $\frac{1}{8}$\\
			\cdashline{2-12}
			& 9 & $p(1,\cdot)$ & $p(2,\cdot,\mathbf{x}_1)$ & $p(3,\cdot, \mathbf{x}_{1:2})$ & $p(4,\cdot,\mathbf{x}_{1:3})$  & $p(5,\cdot,\mathbf{x}_{1:4})$ & $p(6,\cdot,\mathbf{x}_{1:5})$ & $p(7,\cdot,\mathbf{x}_{1:6})$ & $p(8,\cdot,\mathbf{x}_{1:7})$ & \textbf{2} & \textbf{2} \\
			& 10 & $p(1,\cdot)$ & $p(2,\cdot,\mathbf{x}_1)$ & $p(3,\cdot, \mathbf{x}_{1:2})$ & $p(4,\cdot,\mathbf{x}_{1:3})$  & $p(5,\cdot,\mathbf{x}_{1:4})$ & $p(6,\cdot,\mathbf{x}_{1:5})$ & $p(7,\cdot,\mathbf{x}_{1:6})$ & $p(8,\cdot,\mathbf{x}_{1:7})$ & $p(9,\cdot,\mathbf{x}_{1:8})$ & \textbf{2}\\
			\cdashline{2-12}
	\end{tabular}}
	\caption{An aid to the visualisation of the distribution of $\mathbf{x}_n$ given $\boldsymbol{\theta}=\theta$. The bold entries in the table show the values of $\mathbf{x}_n$ which have probability 1 when conditioned on $\boldsymbol{\theta}=\theta$. The non-bold numbers in the table are used to represent that $\mathbf{x}_n$ is a Bernoulli random variable with $\prob(\mathbf{x}_n=1)$ given by the entry. The entries below the diagonal all have distributions given by \eqref{eq:postp}. In general, unlike the other entries in the table, the entries below the diagonal depend on the actual values of earlier observations, and can be calculated as a posterior weighted mixture of the distributions for lower values of $\theta$. Note rows are grouped into pairs by the dashed lines, according to the value of $\lceil\theta/2\rceil$. }
	\label{Table:SMML}
\end{table}

To give some intuition for this problem, consider that the distribution of $\mathbf{x}_{1:(\theta-1)}$ when conditioned on $\boldsymbol{\theta}=\theta$ is by construction identical to its unconditional distribution, by the definition of the function $p$. Note in particular (since it will be important later) that this implies that for any $A,B\subseteq X$:
\begin{equation}\label{eq: condmarg}
\prob(\mathbf{x}_{1:n}\in A|\mathbf{x}_{1:n}\in B)=\prob(\mathbf{x}_{1:n}\in A|\text{$\mathbf{x}_{1:n}\in B$ and $\boldsymbol{\theta}\leq n$}).
\end{equation}
Further, regarding the distribution of $\mathbf{x}$ conditioned on $\boldsymbol{\theta}=\theta$, note that with probability 1 $\mathbf{x}_\theta=2$, and if $\theta$ is odd, then also $\mathbf{x}_{\theta+1}=2$. Finally, if $\theta$ is even the conditional distribution of $\mathbf{x}_{(\theta+1):\infty}$  (or $\mathbf{x}_{(\theta+2):\infty}$ if $\theta$ is odd), is the distribution of a Binomial process with probability parameter $\frac{1}{\theta}$.

This example has been specifically constructed so that the SMML estimates can be calculated using Lemma \ref{L:SMML}, since the problem set-up allows for estimators whose excess expected message length is 0. Consider in particular  the estimator $\hat{\theta}_n(x_{1:n})=n+1$. Then for $n\in\setN$ and $x_{1:n}\in X_{1:n}$,  $\prob(\mathbf{x}_{1:n}=x_{1:n}|\hat{\boldsymbol{\theta}}_n=n+1)=\prob(\mathbf{x}_{1:n}=x_{1:n})=P_{n+1}^{(1:n)}(x_{1:n})$. Therefore, by Lemma \ref{L:SMML}, $\hat{\theta}$ has 0 excess expected message length for all $n\in\setN$, and is therefore a strict minimum message length estimator. Indeed any estimator satisfying
\begin{equation} 
\text{$\hat{\theta}_n(x_{1:n})=f(n)$, where for all $n>2$, $f(n)>n$}\label{eq:thetaf}
\end{equation}
 is an SMML estimator. All such estimators diverge as $n$ goes to infinity, and so are inconsistent.

We show further that there exists \emph{no} consistent SMML estimator on $T$ by showing that the estimators satisfying \eqref{eq:thetaf} are the only SMML estimators.

Suppose $\hat{\theta}$ is an SMML estimator not satisfying \eqref{eq:thetaf}, and suppose $n^*\in\setN$ and $\theta^*\in\Theta$ are such that $\theta^*\leq n^*$ and $\prob(\hat{\boldsymbol{\theta}}_n=\theta^*)>0$. Because there exist estimators with 0 excess expected message length, $\hat{\theta}$, as an SMML estimator, must also have 0 excess expected message length.  Note that since $\theta^*\leq n^*,$ $\supp(P_{\theta^*}^{(1:n)})\subsetneq\supp(\mathbf{x}_{1:n})$ (since for instance, if $x_{\theta^*}\ne 2$ then $x_{1:n}\notin\supp(P_{\theta^*}^{(1:n)})$). Note also that for any $\theta>n^*$ it holds that $\supp(P_\theta^{(1:n)})=\supp(\mathbf{x}_{1:n^*})$ and $\supp(\mathbf{x}_{1:n})\cap\supp\left(P_{\theta^*}^{(1:n)}\right)=\supp\left(P_{\theta^*}^{(1:n)}\right)\ne\emptyset$ and so, by Corollary \ref{C:SMML},
\begin{equation}
\hat{\theta}_{n^*}(x_{1:n^*})\leq n^*\label{eq:lessthann}
\end{equation} for all $x_{1:n^*}\in\supp(\mathbf{x}_{1:n^*})$. 

An estimator with 0 excess expected message length satisfying \eqref{eq:lessthann} is impossible due to the way in which the supports of the distributions for values of $\theta<n^*$ overlap. In particular consider $x^{(1)}_{1:n^*},x^{(2)}_{1:n^*}\in X_{1:n^*}$ where $x^{(1)}_{1:2}=x^{(2)}_{1:2}=\langle 2,2\rangle$ and $x^{(1)}_{3:n^*}=\langle 1,1,\ldots\rangle$ while $x^{(2)}_{3:n^*}=\langle 0,0,\ldots\rangle$. Amongst members of $\Theta$ less than or equal to $n$ the only values with positive likelihood given $x^{(1)}_{1:n^*}$ are $1$ and $2$, while the only value with positive likelihood given $x^{(2)}_{1:n^*}$ is $2$.  Thus, if we let $\hat{\theta}_n^{-1}(\theta)\defeq\{x_{1:n}\in X_{1:n}|\hat{\theta}_n(x_{1:n})=\theta\}$, then again by Corollary \ref{C:SMML}, $\hat{\theta}_{n^*}(x^{(1)}_{1:n^*})=\hat{\theta}_{n^*}(x^{(2)}_{1:n^*})=2$ and $\hat{\theta}_{n^*}^{-1}(2)=\supp(P_2^{(1:n^*)})=\{x_{1:n^*}\in X_{1:n^*}|x_{1:2}=\langle 2,2\rangle\text{ and for all }i>2, x_i\in\{0,1\}\}$. This implies that $\hat{\theta}_{n^*}(\mathbf{x}_{1:n^*})=2$ and $\boldsymbol{\theta}\leq n^*$ if and only if $\boldsymbol{\theta}\in\{1,2\}$.  Therefore using \eqref{eq: condmarg},
\begin{align*}
	&\prob(\mathbf{x}_{1:n^*}=x^{(1)}_{1:n^*}|\hat{\theta}_{n^*}(\mathbf{x}_{1:n^*})=2) \\
&\qquad\qquad=\prob(\mathbf{x}_{1:n^*}=x^{(1)}_{1:n^*}|\text{$\hat{\theta}_{n^*}(\mathbf{x}_{1:n^*})=2$ and $\boldsymbol{\theta}\leq n^*$})\\
&\qquad\qquad=\prob(\mathbf{x}_{1:n^*}=x^{(1)}_{1:n^*}|\text{$\hat{\theta}_{n^*}(\mathbf{x}_{1:n^*})=2$ and $\boldsymbol{\theta}\in\{1,2\}$})\\
&\qquad\qquad=\prob(\mathbf{x}_{1:n^*}=x^{(1)}_{1:n^*}|\boldsymbol{\theta}\in\{1,2\})\\
&\qquad\qquad=\frac{\prob(\boldsymbol{\theta}=1)P_1^{(1:n^*)}\left(x^{(1)}_{1:n^*}\right)+\prob(\boldsymbol{\theta}=2)P_2^{(1:n^*)}\left(x^{(1)}_{1:n^*}\right)}{\prob(\boldsymbol{\theta}\in\{1,2\})}\\
&\qquad\qquad=\frac{\prob(\boldsymbol{\theta}=1)+\prob(\boldsymbol{\theta}=2)\cdot 2^{2-n^*}}{\prob(\boldsymbol{\theta}\in\{1,2\})}\\
&\qquad\qquad=\frac{2}{3}+\frac{1}{3\cdot 2^{n^*-2}}\ne2^{2-n^*} = P_2^{(1:n^*)}(x_{1:n^*}^{(1)}).
\end{align*}
 The last line contradicts the assumption that $I_2(\hat{\theta}_n)=0$ by Lemma \ref{L:SMML}.

We finally show that $T$ has distinctive likelihoods. To this end, it is sufficient to show that for all $\theta\in\Theta$ for $P_\theta$-nearly every $x\in X$: 
\begin{equation*}
\limsup_{n\to\infty}\frac{\sup_{\theta'\in\Theta\setminus\{\theta\}}P_{\theta'}^{(1:n)}(x_{1:n})}{P_\theta(x_{1:n})}<1.
\end{equation*}

Equivalently, we show that both
\begin{equation}\label{Eq:t_low}
\limsup_{n\to\infty}\frac{\sup_{\theta'\le n,\theta'\ne\theta}P_{\theta'}^{(1:n)}(x_{1:n})}{P_\theta(x_{1:n})}<1
\end{equation}
and
\begin{equation}\label{Eq:t_high}
\limsup_{n\to\infty}\frac{\sup_{\theta'> n,\theta'\ne\theta}P_{\theta'}^{(1:n)}(x_{1:n})}{P_\theta(x_{1:n})}<1
\end{equation}
hold true.

For this, let us fix $\theta$, and let us consider only those $x$ that satisfy
for all $n\in\setN$, $x_{1:n}\in\supp(P^{(1:n)}_\theta)$, or, equivalently,
$P^{(1:n)}_\theta(x_{1:n})>0$. This is true for $P_\theta$-nearly
every $x\in X$, so it is enough to prove \eqref{Eq:t_low} and \eqref{Eq:t_high}
under this assumption.

Furthermore, let us consider also only those $n\in\setN$ such that $\theta<n$,
noting that, again, it is enough to prove \eqref{Eq:t_low} and \eqref{Eq:t_high}
under this assumption.

Let $\theta^*$ be the value other than $\theta$ satisfying $\lceil\frac{\theta^*}{2}\rceil=\lceil\frac{\theta}{2}\rceil$.
Among the values that are at most $n$, $\theta^*$ is the only one distinct
from $\theta$ that may potentially have $P_{\theta^*}(x_{1:n})>0$, and for it
we know that
\[
\limsup_{n\to\infty}\frac{P_{\theta^*}(x_{k+1:n})}{P_{\theta}(x_{k+1:n})}=0,
\]
where $k=\max(\theta,\theta^*)$, due to Lemma~\ref{L:Bin_CLR}, because from
$\mathbf{x}_{k+1}$ on, all $\mathbf{x}_i$ values are independent and identically
Bernoulli-distributed given $\boldsymbol{\theta}$.

Because $x_{1:n}$ has been assumed to be in $\supp(P_\theta^{(1:n)})$
for all $n\in\setN$, we know that $P_\theta^{(1:k)}(x_{1:k})>0$. Hence,
\begin{equation*}
 	\limsup_{n\to\infty} \frac{P_{\theta^*}^{(1:n)}(x_{1:n})}{P_\theta^{(1:n)}(x_{1:n})} = \limsup_{n\to\infty} \left(\frac{P_{\theta^*}^{(1:k)}(x_{1:k})}{P_\theta^{(1:k)}(x_{1:k})}\right)\left(\frac{P_{\theta^*}^{((k+1):n)}(x_{(k+1):n})}{P_\theta^{((k+1):n)}(x_{(k+1):n})}\right) = 0,
\end{equation*}
proving \eqref{Eq:t_low}.

We now proceed to show \eqref{Eq:t_high}.

Recall that for any $\theta'>n$, by construction
\[
P_{\theta'}^{(1:n)}(x_{1:n})=\prob(\mathbf{x}_{1:n}=x_{1:n})=\prob(\mathbf{x}_{1:n}=x_{1:n}|\boldsymbol{\theta}\le n).
\]

Combining this with \eqref{Eq:t_low}, we conclude that
\begin{align*}
\limsup_{n\to\infty}\frac{\sup_{\theta'>n}P_{\theta'}^{(1:n)}(x_{1:n})}{P_\theta(x_{1:n})}
&=\limsup_{n\to\infty}\frac{\prob(\mathbf{x}_{1:n}=x_{1:n}|\boldsymbol{\theta}\le n)}{P_\theta^{(1:n)}(x_{1:n})}\\
&=\limsup_{n\to\infty}\sum_{\theta'=1}^{n} \frac{\prob(\boldsymbol{\theta}=\theta')P^{(1:n)}_{\theta'}(x_{1:n})}{\prob(\boldsymbol{\theta}\le n) P_\theta^{(1:n)}(x_{1:n})}\\
&<\limsup_{n\to\infty}\sum_{\theta'=1}^{n} \frac{\prob(\boldsymbol{\theta}=\theta')}{\prob(\boldsymbol{\theta}\le n)}=1,
\end{align*}
proving \eqref{Eq:t_high}, and therefore the entire claim.
\end{proof}

\end{document}